\documentclass{amsart}
\usepackage{amsmath,amsfonts,amsthm,amssymb}
\newtheorem{lemma}{Lemma}[section]

\newtheorem{corollary}[lemma]{Corollary}
\newtheorem{theorem}[lemma]{Theorem}
\newtheorem{definition}[lemma]{Definition}
\newtheorem{example}[lemma]{Example}

\def\jca{\mathcal J}
\def\tta{\tilde\tau}
\title{Effect algebras with state operator}
\author{A. Jen\v cov\'a, S. Pulmannov\'a}\thanks{This work was supported by grant VEGA 2/0059/12 and by Science and Technology Assistance Agency under the contract no. APVV-0178-11.}

\date{}

\address{Mathematical Institute, Slovak Academy of Sciences, \v
Stef\'anikova 49, 814 73 Bratislava, Slovakia }
\email{pulmann@mat.savba.sk, jenca@mat.savba.sk}



\begin{document}
\maketitle

\begin{abstract} State operators on convex effect algebras, in particular effect algebras of unital JC-algebras, JW-algebras and convex $\sigma$-MV algebras are studied and their relations with conditional expectations in algebraic sense as well as in the sense of probability on MV-algebras are shown.

\end{abstract}

\section{Introduction}

Effect algebras have been introduced by Foulis and Bennett \cite{FoBe} (see also \cite{GiGr, KCh} for equivalent definitions) for modeling unsharp measurements in quantum mechanical systems \cite{BLM}. They are a generalization of
many structures which arise in the axiomatization of quantum mechanics (Hilbert space effects \cite{Lud}), noncommutative measure theory and probability (orthomodular lattices and posets, \cite{BeCa, PtPu}), fuzzy measure theory and many-valued logic (MV-algebras \cite{Chang, CDM}).

A state, as an analogue of a probability measure, is a basic notion in algebraic structures used in the quantum theories (see e.g., \cite{DvPu}), and properties of states have been deeply studied by many authors.

In MV-algebras, states as averaging the truth value  were first studied in \cite{Mustate}. In the last few years, the notion of a state has been studied by many experts in MV-algebras, e.g, \cite{RiMu, KM}.

Another approach to the state theory on MV-algebras has been presented recently in \cite{FM}. Namely, a new unary operator  was added to the MV-algebras structure as an internal state (or so-called state operator). MV-algebras with the added state operator are called state MV-algebras.   The idea is that an internal state has some properties reminiscent of states, but, while a state is a map from an MV-algebra into $[0,1]$, an internal state is an operator of the algebra.  State MV-algebras generalize, for example, H\'ajek's approach \cite{Haj} to fuzzy logic with modality $Pr$ (interpreted as {\it probably}) with the following semantic interpretation: The probability of an event $a$ is presented as the truth value of $Pr(a)$. For a more detailed motivation of state MV-algebras and their relation to logic, see \cite{FM}.

In \cite{BCD}, the notion of a state operator was extended from MV-algebras to the more general frame of effect algebras. A state operator is there defined as an additive, unital and idempotent operator on $E$.
A state operator on $E$ is called strong, if it satisfies the additional condition
\begin{equation}
\tau(\tau(a)\wedge \tau(b))=\tau(a)\wedge \tau(b) \, \mbox{whenever}\  \tau(a)\wedge \tau(b) \ \mbox{exists in} \ E.
\end{equation}
Since MV-algebras form a special subclass of effect algebras, so-called MV-effect algebras, it was shown that
the definition of a state operator on effect algebras coincides with the original definition on MV-algebras if and only if the state operator is strong. Moreover,  if $\tau$ is faithful, i.e., $\tau(a)=0$ implies $a=0$, then it is automatically strong.

In the present paper, we show that state operators  on an effect algebra $E$ are related with states on $E$ in the following way: (1)  every state on $E$ induces a state operator on the tensor product $[0,1]\otimes E$. (2) If $E$ admits an ordering set of states, then every state operator on $E$ induces a state on $E$.
We study state operators mainly on convex effect algebras.
Convex effect algebras as effect algebras with an additional convexity structure were introduced and studied in \cite{GuPurepr, BBGP, BGP}. It was proved in \cite{GuPurepr}, that every convex effect algebra  is isomorphic with the interval $[0,u]$ in an ordered real linear space $(V,V^+)$, where $u$ is an order unit. Moreover, $(V,u)$ is an order unit space, i.e. $u$ is an archimedean order unit, if and only if $E$ admits an ordering set of states \cite{BBGP}.
Using the tensor product of an effect algebra with the interval $[0,1]$ of reals, we show that every effect algebra $E$ admitting at least one state, can be embedded into a convex one.   Moreover, a state operator on $E$ extends to a state operator on its convex envelope. It is therefore not too restrictive to concentrate our interest on convex effect algebras with state operators.
We show that a state operator on a convex effect algebra is an affine mapping which extends to a linear, positive and idempotent mapping of the corresponding ordered linear space into itself. Moreover, the state operator is faithful if and only if its extension is faithful.

A prototype of an effect algebra is the set of Hilbert space effects ${\mathcal E}(H)$, i.e., self-adjoint operators between the zero and identity operator on a (complex) Hilbert space $H$ with respect to the usual ordering of self-adjoint operators.  The partial operation $\oplus$ is defined as the usual operator sum of two effects  whenever this sum is also an effect, and the effect algebra ordering coincides with the original one we started with.   This effect algebra is naturally convex. It plays an important role in quantum measurement theory, because the most general quantum observables, positive operator valued measures (POVMs), have their ranges in it \cite{BLM}.  On the set ${\mathcal B}(H)$ of the bounded operators on $H$, we consider the von Neumann  \cite{vN} and L\"uders \cite{Lu} conditional expectations, and we show that their restrictions to ${\mathcal E}(H)$ are faithful, hence strong state operators. Motivated by these examples, we study relations between state operators and conditional expectations on so-called JC-effect algebras.

Recall that a JC-algebra $\jca$ is a norm-closed real vector subspace of bounded self-adjoint operators on a Hilbert space $\mathcal H$, closed under the Jordan product $a\circ b=\tfrac 12(ab+ba)$. A JC-algebra is called a JW-algebra if it is closed in the weak topology \cite{Top}. We study JC-algebras containing a unit element $1$, and the interval $[0,1]$ is then called a \emph{JC-effect algebra}.   Let $\tau :{\mathcal E}({\jca})\to {\mathcal E}(\jca)$ be a state operator. Since $\mathcal E(\jca)$  is the $[0,1]$ interval in the ordered vector space $(\jca,\jca^+)$, $\tau$ extends to a linear, positive idempotent and unital mapping $\tilde \tau: \jca\to \jca$. Such maps were studied in several papers, we use mostly the results in \cite{ES}.

We call a state operator $\tau :{\mathcal E}({\jca})\to {\mathcal E}(\jca)$ a \emph{conditional expectation} iff it has the property $\tau(\tau(a)b\tau(a))=\tau(a)\tau(b)\tau(a)$ for all $a,b\in {\mathcal E}(\jca)$. We show that a
state operator $\tau$ on ${\mathcal E}(\jca)$ is a conditional expectation iff its range is a JC-sub-effect algebra  of ${\mathcal E}(\jca)$, and any faithful state operator is a conditional expectation. Moreover, a state operator can be expressed as a combination of a conditional expectation and so-called Jordan state operator.
We also show that if the JC-effect algebra is the unit interval in $C(X;{\mathbb R})$, the set of real-valued continuous functions on a compact Hausdorff space $X$, then a state $\tau$ operator is a conditional expectation iff  $\tau$ is strong.

In the probability theory on MV-algebras, an MV-conditional expectation on a  $\sigma$-MV-algebra $M$ with respect to a $\sigma$-MV-subalgebra $N$ of $M$ in a $\sigma$-additive state $m$ on $M$  was introduced in \cite{DvPucond}.  We study relations between the MV-conditional expectations  and state operators on convex $\sigma$-MV-algebras.
Since a convex $\sigma$-MV-algebra $M$ is isomorphic with the interval $[0,1]$ in $C(X;{\mathbb R})$ (\cite[Theorem 7.3.12]{DvPu}, a state operator $\tau$ on $M$ is a conditional expectation iff $\tau$ is strong.

We show that an MV-conditional expectation induces a strong state operator on the quotient of $M$ with respect to the kernel of the state $m$.

On the other hand, a strong  $\sigma$-additive state operator (hence a conditional expectation) $\tau$ on $M$ induces an MV-conditional expectation on $M$ with respect to the sub-MV-algebra equal to the range of $\tau$ and in states of the form $s\circ\tau$, where $s$ is any $\sigma$-additive state.

\section{Basic definitions}
\begin{definition} \label{df:EA}
An \emph{effect algebra} (EA, for short) is a structure $(E;0,1,\oplus)$
consisting of a set $E$; elements $0$ and $1$ in $E$ called the
\emph{zero} and the \emph{unit}; and a partially defined binary operation
$\oplus$ on $E$ called the \emph{orthosummation}, such that, for all
$d,e,f\in E$:
\begin{enumerate}
\item[(EA1)] If $e\oplus f$ is defined, then $f\oplus e$ is defined and
 $e\oplus f=f\oplus e$.
\item[(EA2)] If $d\oplus e$ and $(d\oplus e)\oplus f$ are defined, then
 $e\oplus f$ and $d\oplus(e\oplus f)$ are defined and $(d\oplus
 e)\oplus f=d\oplus(e\oplus f)$.
\item[(EA3)] For each $e\in E$ there exists a unique element $e\sp{\perp}
 \in E$,  called the \emph{orthosupplement} of $e$, such that
 $e\oplus e\sp{\perp}=1$.\footnote{If we write an equation involving
 an orthosum $e\oplus f$ without explicitly stating that $e\oplus f$ is
 defined, we are tacitly assuming that it is defined.}
\item[(EA4)] $e\oplus 1$ is defined only if $e=0$.
\end{enumerate}
\end{definition}

In an effect algebra $E$, a partial ordering is defined by $a\leq b$ iff there is $c\in E$ such that $a\oplus c=b$.
It turns out that the element $c$, if it exists, is uniquely defined. This enables us to introduce the partial operation $\ominus$ by $b\ominus a:=c$ \footnote{The notation $:=$ means `equals by definition'.} iff $a\oplus c=b$. Clearly, $b\ominus a$ is defined iff $a\leq b$.
With respect to this partial order we have $0\leq a\leq 1$ for all $a\in E$.

Effect algebras were introduced in \cite{FoBe}. Notice that effect algebras are equivalent to D-posets \cite{KCh}, which are partial algebraic structures based on the operation $\ominus$ \cite{DvPu}.

The operation $\oplus$ can be extended to suitable sequences $a_1,a_2,\ldots,a_n$ of elements (not necessarily all different) by recurrence. We say that the elements $a_1,a_2,\ldots,a_n$ are \emph{orthogonal} iff their orthosum $a_1\oplus a_2\oplus \cdots \oplus a_n=:\oplus_{i\leq n}a_i$ is defined. An arbitrary family of elements $\{ a_{\lambda}:\lambda \in I\}$ in $E$ is called \emph{orthogonal}, iff every its finite subfamily is orthogonal. Let ${\mathcal F}(I)$ denote the set of all finite subsets of the index set $I$. If the element $a:=\bigvee_{F\in {\mathcal F}(I)}\oplus_{i\in F}a_i$ exists, then $a$ is called the \emph{orthosum} of the orthogonal family $\{ a_{\lambda}:\lambda \in I\}$. The effect algebra $E$ is called \emph{$\sigma$-orthocomplete} (\emph{orthocomplete})
iff every countable (arbitrary) orthogonal family has an orthosum.

\begin{example} \label{ex:booleanalg}\rm{

An orthomodular lattice $(L; \leq, ^{\perp}, 0,1)$ is organized into an EA by defining $a\oplus b$
for $a,b\in L$ iff $a\leq b^{\perp}$, in which case $a\oplus b :=
a\vee b$.}

Conversely, an effect algebra $E$ is an OML iff it is lattice ordered and  $a\leq b^{\perp} \,\implies\, a\wedge b=0$.
\end{example}

\begin{example}\label{ex:MValg}
\rm{Recall that an \emph{MV-algebra} is a system $(M;\boxplus,',0)$, where $(M;\boxplus,0)$ is a commutative monoid with neutral element $0$, and for each $x,y\in M$ the following equations hold:
\begin{enumerate}
\item[(i)] $(x')'=x$,
\item[(ii)] $x\boxplus1=1$, where $1=0'$,
\item[(iii)] $x\boxplus(x\boxplus y')'=y\boxplus(y\boxplus x')'$.
\end{enumerate}
In every MV-algebra one can define further operations as follows:
\begin{eqnarray*}
x\boxdot y&=&(x'\boxplus y')',\, x\vee y=(x'\boxplus y)'\boxplus y \\
x\wedge y&=&(x'\vee y')', \, x\boxminus y=(x'\boxplus y)'.
\end{eqnarray*}

A prototype of an MV-algebra is the unit interval $[0,1]\subseteq {\mathbb R}$, with operations
$x\boxplus y=\min\{1,x+y\}$, and $x'=1-x$.
An MV-algebra becomes a boolean algebra iff the following identity holds: $a\boxplus a=a$.

An MV-algebra $(M; \boxplus, ',0)$ is organized into an EA by defining
$a\oplus b$ for $a,b\in M$ iff $a\boxdot b=0$, in which case $a\oplus b:=a\boxplus b$.

Conversely, an EA $E$ can be made an MV-algebra iff $E$ is a lattice and

$$
 a\wedge b=0 \implies \,  a\leq b^{\perp},
$$
in which case $E$ is called an \emph{MV-effect algebra}. The MV-algebra operations are defined by $a\boxplus b:=a\oplus (a^{\perp}\wedge b)$ and $a':=a^{\perp}$.

 MV-algebras and MV-effect algebras are in one-to one correspondence, and we identify them.

 Notice that Boolean algebras coincide with the subclass of MV-algebras satisfying the additional identity $a\boxplus a=a$.}

\end{example}

\begin{example}\label{ex:effects}\rm{Let $H$ be a (complex) Hilbert space. A (self-adjoint) operator  $a:H\to H$ is called an \emph{effect} iff $0\leq a\leq I$ (with respect to the usual ordering of self-adjoint operators).
Let ${\mathcal E}(H)$ denote the set of all effects on $H$, then ${\mathcal E}(H)$ can be organized into an effect algebra with constants $0$ and $I$ as zero and unit, respectively,  by putting $a\oplus b$ is defined iff $a+b$ is an effect, and in this case, $a\oplus b=a+b$. The effect algebra ordering coincides with the original one that we used in the definition of effects.}
\end{example}

\begin{example}\label{ex:group}\rm{ Let $(G;\leq, +,0)$ be an abelian partially ordered group. Then for any $u\in G, u\geq 0$, the interval $G[0,u]:=\{ x\in G: 0\leq x\leq u\}$ can be organized into an effect algebra by defining $a\oplus b=a+b$, provided that $a+b\leq u$. The effect algebra ordering coincides with the restriction of the partial order in $G$ to the interval $G[0,u]$. Effect algebras of this type are called \emph{interval effect algebras}.

Notice that ${\mathcal E}(H)$ is an interval effect algebra. Indeed, ${\mathcal E}(H)={\mathcal B}(H)^{sa}[0,I]$, where ${\mathcal B}(H)^{sa}$ is the self-adjoint part of ${\mathcal B}(H)$.

Moreover, by \cite{Mu}, MV-algebras are equivalent with unit intervals in abelian  $\ell$-groups with a strong unit.

On the other hand, there are examples of OMLs that are not interval effect algebras \cite{Na,We}.
}
\end{example}

A subset $F$ of an effect algebra $E$ is a \emph{subeffect algebra} of $E$ iff (i) $a\in  F$ implies $a^{\perp}\in F$, (ii) $a,b\in F$ and $a\oplus b$ exists in $E$ implies $a\oplus b\in F$.

Let $E$ be an effect algebra and $a\in E$. The interval $E[0,a]:=\{ x\in E: x\leq a\}$ with the operation $\oplus$ restricted to $E[0,a]$ (i.e., $x\oplus_a y$ exists if $x\oplus y$ exists in $E$ and $x\oplus y\leq a$) is an effect algebra with unit element $a$.

Let $E,F$ be two EAs. A mapping $\phi :E\to F$ is said to be:

(i) a \emph{morphism} iff $\phi(1)=1$, and $p\perp q, p,q\in E$ implies $\phi(p)\perp \phi(q)$, and $\phi(p\oplus q)=\phi(p)\oplus \phi(q)$;

(ii) a \emph{monomorphism} iff $\phi$ is a morphism and $\phi(p)\perp \phi(q)$ iff $p\perp q$;

(iii) an \emph{isomorphism} iff $\phi$ is a surjective monomorphism, and we say that $E$ is \emph{isomorphic} to $F$;

(iv) a \emph{state} iff $\phi$ is a morphism and $F$ is the MV effect algebra $[0,1]$.

If $E$ and $F$ are $\sigma$-orthocomplete (orthocomplete) effect algebras, then a morphism $\phi :E\to F$ is a $\sigma$-morphism (complete morphism) iff
$\phi(\oplus_{i\in {\mathbb N}}a_i)=\oplus_{i\in {\mathbb N}}\phi(a_i)$, whenever $\oplus_{i\in {\mathbb N}}a_i$ exists in $E$ (iff $\phi(\oplus_{i\in I}a_i)=\oplus_{i\in I}\phi(a_i)$ whenever $\oplus_{i\in I}a_i$ exists for an arbitrary index set $I$).

Let $E,F,L$ be EAs. A mapping $\beta:E\times F \to L$ is called a \emph{bimorphism} iff:

(i) $a,b\in E$ with $a\perp b$, $q\in F$ imply $\beta(a,q)\perp \beta(b,q)$ and $\beta(a\oplus b,q)=\beta(a,q)\oplus \beta(b,q)$;

(ii) $c,d\in F$ with $c\perp d$, $p\in E$ imply $\beta(p,c)\perp \beta(p,d)$ and $\beta(p,c\oplus d)=\beta(p,c)\oplus \beta(p,d)$;

(iii) $\beta(1,1)=1$.

If, in addition, $E$ and $F$ are MV-effect algebras, then also the following properties are required:

(iv) $a_1,a_2\in E$, $b\in F$ imply $\beta(a_1\vee a_2,b)=\beta(a_1,b)\vee \beta(a_2,b)$,  $\beta(a_1\wedge a_2)=\beta(a_1b)\wedge \beta(a_2,b)$;

(v) $a\in E$, $b_1,b_2\in F$ imply $\beta(a,b_1\vee b_2)=\beta(a,b_1)\vee \beta(a,b_2)$, $\beta(a,b_1\wedge b_2)=\beta(a,b_1)\wedge \beta(a,b_2)$.

Recall that a bimorphism $\beta:P\times Q \to L$, where $P,Q$ and $L$ are $\sigma$-effect algebras ($\sigma$-MV-effect algebras), is a $\sigma$-bimorphism iff whenever $a_i\in P, b_i\in Q$ are increasing then $\beta(\bigvee a_i,b)=\bigvee \beta(a_i,b)$ for every $b\in Q$, and $\beta(a,\bigvee b_i)=\bigvee(a,b_i)$ for every $a\in P$.

The following is a generalization of the definition of tensor product of effect algebras in \cite{Dv95}.

\begin{definition}\label{de:tensprod} Let ${\mathcal K}$ be a class of effect algebras and $E,F\in {\mathcal K}$. We say that a pair $(T,\rho)$ consisting of an effect algebra $T\in {\mathcal K}$ and a ${\mathcal K}$- bimporphism $\rho:E\times F \to T$ is a \emph{tensor product} of $E$ and $F$ in the class ${\mathcal K}$ iff the following condition is satisfied:

(T) If $L\in {\mathcal K}$ and $\beta:E\times F\to L$ is a ${\mathcal K}$- bimorphism, there exists a unique  ${\mathcal K}$-morphism $\phi :T\to L$ such that $\beta =\phi \circ \rho$.

\end{definition}

The following theorem was proved in \cite[Theorem 5.3]{Dvijtp}.

\begin{theorem}\label{th:tenprod}
Let both effect algebras $P$ and $Q$ posses  at least one state. Then
the tensor product  $P\otimes Q$ of $P$ and $Q$ exists in the category of effect algebras. In addition,
for any state $\mu$ on $P$ and any state $\nu$ on $Q$ there is a unique state $\mu\otimes\nu$ on
$P\otimes Q$ such that $\mu\otimes \nu(p\otimes q)= \mu(p)\nu(q)$, $p\in P$, $q\in Q$.
\end{theorem}

Tensor products in the  category of $\sigma$-effect algebras were studied in \cite{Gsigma}. Tensor products in the category of MV-algebras (equivalently, MV-effect algebras) were studied in \cite{MuLS}.

Notice that for the unit interval $[0,1]\subseteq {\mathbb R}$, in the category of $\sigma$-effect algebras,  $[0,1]\otimes_{\sigma}[0,1]=[0,1]$, where the bimorphism is given by $(\alpha,\beta)\mapsto \alpha \beta$ \cite{Gsigma}, while in the category of effect algebras, $[0,1]\otimes[0,1]\neq [0,1]$ \cite{Ptenprod}.

\section{State operator}

Recently, in \cite{FM}, the notion of an MV-algebra with internal state, \emph{state MV-algebra} (SMV-algebra for short), was introduced as follows. An SMV-algebra is a structure $(M,\sigma)= (M;\boxplus,',\sigma,0)$, where $(M;\boxplus,',0)$ is an MV-algebra, and $\sigma$ is a unary operator on $M$ satisfying, for each $x,y\in M$:
\begin{enumerate}
\item[($\sigma$1)] $\sigma(0)=0$.
\item[($\sigma$2)] $\sigma(x')=(\sigma(x))'$.
\item[($\sigma$3)] $\sigma(x\boxplus y)=\sigma(x)\boxplus\sigma(y\boxminus(x\boxdot y))$.
\item[($\sigma$4)] $\sigma(\sigma(x)\boxplus \sigma(y))=\sigma(x)\boxplus \sigma(y)$.
\end{enumerate}

 In \cite{BCD},the notion of an internal state (called also a state operator) was extended to the more general  frame of effect algebras.

\begin{definition}\label{de:stateop} A \emph{state operator} on an EA $E$ is a mapping $\tau :E\to E$ such that, for all $e,f\in E$,
\begin{enumerate}
\item[(i)] $\tau(1)=1$;
\item[(ii)] $\tau(e\oplus f)=\tau(e)\oplus \tau(f)$;
\item[(iii)] $\tau(\tau(a))=\tau(a)$.
\end{enumerate}
The couple $(E;\tau)$, where $E$ is an effect algebra and $\tau$ is a state operator, will be called a \emph{state effect algebra} (SEA).
\end{definition}

\begin{lemma}\label{le:propoftau} For every SEA the following properties hold:

(i) $\tau(0)=0$.

(ii) $\tau(a^{\perp})=(\tau(a))^{\perp}$.

(iii) $a\leq b \ \implies\ \tau(a)\leq \tau(b)$, and $\tau(b\ominus a)=\tau(b)\ominus \tau(a)$.

(iv) If $a\wedge b$ exists, then $\tau(a\wedge b)\leq \tau(a), \tau(b)$, and if $a\vee b$ exists, then $\tau(a),\tau(b)\leq \tau(a\vee b)$.

(v) $\tau(E)$ is a subeffect algebra of $E$.
\end{lemma}

\begin{lemma} {\rm \cite{BCD}} A state operator on an MV-effect  algebra $(E;\oplus, 0,1)$ is a state operator on the MV-algebra $(E; \boxplus,',0)$ in the sense of \cite{FM} if and only if the following additional condition is satisfied:
\begin{equation}\label{eq: x}
\exists  \tau(e)\wedge \tau(f) \,\implies \, \tau(\tau(e)\wedge \tau(f))=\tau(e)\wedge \tau(f).
\end{equation}
\end{lemma}
\begin{proof} Let $(E;\boxplus,',0,1)$ be the MV-algebra corresponding to the MV-effect algebra $(E;\oplus, 0,1)$.
Then $\tau(0)=0$, $\tau(a')=\tau(a)'$ by Lemma \ref{le:propoftau}. Moreover, for any $a,b\in E$, $\tau(a\boxplus b)=\tau(a\oplus a^{\perp}\wedge b)=\tau(a)\oplus\tau(a^{\perp}\wedge b)=\tau(a)\boxplus\tau((a\vee b')')=\tau(a)\boxplus\tau((a'\boxplus b')'\boxplus b')')=\tau(a)\boxplus\tau(b\boxminus(a\odot b))$, as required.
Finally, $\tau(\tau(a)\boxplus\tau(b))=\tau(\tau(a)\oplus \tau(a)^{\perp}\wedge \tau(b))=\tau(a)\oplus\tau( (\tau(a^{\perp})\wedge \tau(b))=\tau(a)\oplus \tau(a)^{\perp}\wedge \tau(b)=\tau(a)\boxplus\tau(b)$.

The converse statement follows from the fact that if $\sigma$ is an internal state on the MV-algebra $(E;\boxplus,',0,1)$, then the range of $\sigma$ is MV-subalgebra of $E$.
\end{proof}

In \cite{BCD}, a state operator $\tau$ satisfying property (\ref{eq: x}), is called a \emph{strong state operator}.
Notice that an effect algebra need not be a lattice, in general, and so condition (\ref{eq: x}) might be difficult to check. Nevertheless, the following lemma shows an important situation in which this condition always holds.

\begin{definition}\label{de:faithful} A state operator $\tau$ on $E$ is \emph{faithful} if for any $a\in E$, $\tau(a)=0$ implies $a=0$.
\end{definition}

\begin{lemma}\label{le:faithful}{\rm \cite{BCD}} If $\tau$ be a faithful state operator on an effect algebra $E$, then $\tau$ is strong.
\end{lemma}
\begin{proof} Let $a,b\in E$ and assume that $c:=\tau(a)\wedge \tau(b)$ exists in $E$.  From $c\leq \tau(a)$ we get $\tau(c)\leq \tau(\tau(a))=\tau(a)$, and similarly, $\tau(c)\leq \tau(b)$, hence $\tau(c)\leq c$. This entails $\tau(c\ominus \tau(c))=0$, and as $\tau$ is faithful, we get $c=\tau(c)$, which is (\ref{eq: x}).
\end{proof}

In the next theorem, we show that every state operator on an MV-effect algebra $M$ induces a faithful state operator on a quotient of $M$.

Recall that a subset $I$ of an effect algebra $E$ is an \emph{ideal} iff (1) $a\in I$ and $b\leq a$ imply $b\in I$, and (2) $a,b\in I$ and $a\perp b$ imply $a\oplus b\in I$. By \cite[Theorem 3.1.32]{DvPu}, a subset $I$ of an MV-effect algebra is an MV-algebra ideal iff it is an effect algebra ideal. Consequently, the quotient $M|I:=\{ [a]:a\in M\}$, where $[a]$ denote the congruence class containing $a$, $a\in M$, is an MV-effect algebra.

\begin{theorem}\label{th:quotfaithful} Let $\tau$ be a state operator on an MV-effect algebra $M$. Put $I_{\tau}:=\{ a\in M: \tau(a)=0\}$, then $I_{\tau}$ is an MV-algebra ideal, and $\hat{\tau}:M|I_{\tau} \to M|I_{\tau}$ defined by $\hat{\tau}[a]:=[\tau(a)]$ is a faithful state operator on $M|I_{\tau}$.
\end{theorem}
\begin{proof} It is easy to check that $I_{\tau}$ is an effect algebra ideal, hence it is also an MV-algebra ideal, and $M|I_{\tau}$ is an MV-effect algebra. Define $\hat{\tau}[a]:=[\tau(a)]$. We have $a\sim b$ iff $a\Delta b\in I_{\tau}$, where $a\Delta b=(a\vee b)\ominus(a\wedge b)$ is the symmetric difference \cite[Theorem 2.2.21 (v)]{DvPu}.
Hence $a\sim b$ iff $\tau(a\Delta b)=0$ and  $\tau(a\vee b)\geq \tau(a), \tau(b)\geq \tau(a\wedge b)$ yields $\tau(a)=\tau(b)$. This proves that $\hat{\tau}$ is well defined. Clearly, $\hat{\tau}: M|I_{\tau} \to M|I_{\tau}$, and $\hat{\tau}[1]=[\tau(1)]=[1]$. Moreover, $[a]\perp [b]$ iff there are $a_1\sim a$, $b_1\sim b$ with $a_1\perp b_1$ and  $[a]\oplus [b]=[a_1\oplus b_1]$. Therefore, $\hat{\tau}([a]\oplus [b])= \hat{\tau}([a_1\oplus b_1])=[\tau(a_1\oplus b_1])=[\tau(a_1)]\oplus [\tau(b_1)]=\hat{\tau}([a])\oplus \hat{\tau}([b])$, so that $\hat{\tau}$ is additive. Finally, $\hat{\tau}(\hat{\tau}[a])=\hat{\tau}([\tau(a)])=[\tau(\tau(a))]=[\tau(a)]=\hat{\tau}[a]$, hence $\hat{\tau}$ is idempotent, and hence a state operator.

Moreover, if $\hat{\tau}([a])=0$, then $[\tau(a)]=0$, hence $\tau(a)\in I$. But then $\tau(\tau(a))=\tau(a)=0$. It follows that $a\in I_{\tau}$, whence $[a]=0$. Consequently, $\hat{\tau}$ is faithful.
\end{proof}

In the rest of this section we show relations between states on effect algebras and their state operators.

\begin{theorem}\label{th:st-stop} Every state on an effect algebra $E$ induces a state operator on the tensor product $[0,1]\otimes E$.
\end{theorem}
\begin{proof} Let $s:E\to [0,1]$ be a state on $E$. Since the identity mapping $i:[0,1]\to [0,1]$ is (the unique) state on $[0,1]$, by \cite{Dvijtp}, the tensor product $[0,1]\otimes E$ exists. Moreover, there is a unique state $\tau:=i\otimes s$ on $[0,1]\otimes E$ with $\tau(\alpha \otimes a)=\alpha s(a)$.  It is easy to check that the mapping $\alpha \mapsto \alpha \otimes 1$ is an isomorphism between $[0,1]$ and $[0,1]\otimes 1\subseteq [0,1]\otimes E$. Thus $\tau(\tau(\alpha\otimes a))=\tau(\alpha s(a)\otimes 1)=\alpha s(a)\otimes 1$.  It follows that $\tau$  is a state operator on $[0,1]\otimes E$.
\end{proof}

\begin{definition}\label{de:ordering} A set $\Omega$ of states on an EA $E$ is called \emph{ordering} iff
for all $e,f\in E$, $\omega(e)\leq \omega(f)$ $\forall \omega \in \Omega$ implies $e\leq f$.
\end{definition}

\begin{theorem}\label{th:ordering}{\rm \cite{Gudnadeq}} An effect algebra $E$ admits an ordering set $\Omega$ of states if and only if $E$ is isomorphic to an effect subalgebra of $[0,1]^{\Omega}$.
\end{theorem}

\begin{theorem}\label{th:stop-st} Every state operator on an effect algebra $E$ with ordering set of states induces a state on $E$.
\end{theorem}
\begin{proof} Let $\Omega$ be the ordering set of states on $E$ and let $\tau:E\to E$ be a state operator on $E$.
Since $\tau(E)$ is a subeffect algebra of $E$, $\Omega$ is an ordering set of states also on $\tau(E)$. By Theorem \ref{th:ordering}, there is a monomorphism $\phi:\tau(E)\to [0,1]^{\Omega}$. Put $M:=[0,1]^{\Omega}$, then $M$ is an MV-algebra, and for a maximal ideal $I$ of $M$, the quotient $M/I$ is isomorphic to a subalgebra of $[0,1]$ (e.g., \cite[Prop. 2.2.33]{DvPu}). Let $q:M\to M/I$ be the quotient mapping. Define $s(a):=q(\phi(\tau(a)))$, $a\in E$.
From the properties of the mappings $\tau, \phi$ and $q$ we easily derive that $s$ is a state on $E$.
\end{proof}

\section{Convex effect algebras}

An effect algebra $E$ is \emph{convex} \cite{GuPurepr} iff for every $a\in E$ and $\lambda \in [0,1]\subseteq {\mathbb R}$ there exists an element $\lambda a \in E$ such that the following conditions hold:
\begin{enumerate}
\item[(C1)] If $\alpha, \beta \in [0,1]$ and $a\in E$, then $\alpha(\beta a)=(\alpha \beta)a$.
\item[(C2)] If $\alpha, \beta \in [0,1]$ with $\alpha +\beta \leq 1$ and $a\in E$, then $\alpha a\perp \beta a$ and $(\alpha +\beta)a=\alpha a\oplus \beta a$.
\item[(C3)] If $a,b\in E$ with $a\perp b$ and $\lambda \in [0,1]$, then $\lambda a\perp \lambda b$ and $\lambda(a\oplus b)=\lambda a\oplus \lambda b$.
\item[(C4)] If $a\in E$, then $1a=a$.
\end{enumerate}

A map $(\lambda,a)\to \lambda a$ that satisfies (C1)--(C4) is called a \emph{convex structure} on $E$. Notice that $0a=0$ for every $a\in E$. Observe that a convex structure on $E$ is bimorphism from $[0,1]\times E$ into $E$.

\begin{example}\label{ex:tenprod}{\rm  It can be shown that every effect algebra which has at least one state can be embedded into a convex one. Indeed, let $E$ be an effect algebra with at least one state. Then the tensor product $(T,\rho)$ of $[0,1]$ and $E$ exists.
Similarly, the tensor product $(T_{\alpha},\rho_{\alpha})$ of $[0,\alpha ]$ and $E$ exists for every $0 \neq \alpha \in [0,1]$. Define the mapping $i_{\alpha} :[0,\alpha ]\times E \to T$ by $i_{\alpha}(\lambda,a)=\rho(\lambda,a)$. The image of $i_{\alpha}$ is in the interval $[0,\rho(\alpha,1)]$ of $T$, and
$i_{\alpha}:[0,\alpha ]\times E \to [0, \rho(\alpha,1)]$ is a bimorphism.  Therefore there is a morphism $\psi_{\alpha}:T_{\alpha}\to[0, \rho(\alpha,1)]$ such that $\psi_{\alpha}\circ \rho_{\alpha}=i_{\alpha}$. Then $\psi_{\alpha}(\rho_{\alpha}(\lambda,a))=\rho(\lambda,a)$, and since $\rho_{\alpha}(\lambda,a)$ are generating elements of $T_{\alpha}$, while $\rho(\lambda,a)$ are generating elements of the interval $[0, \rho(\alpha,1)]$ in $T$ ($\lambda \in [0,\alpha]$, $a\in E$), it follows that $\psi_{\alpha}$ is an isomorphism of $T_{\alpha}$ onto the  interval $[0,\rho(\alpha,1)]$ in $T$.

Define, for every $\alpha \in [0,1]$, the mapping $j_{\alpha}:[0,1]\times E \to T_{\alpha}$, $j_{\alpha}(\lambda,a)=\rho(\alpha \lambda,a)$. This mapping is a bimorphism, and hence it extends to a unique morphism $\phi_{\alpha}:T\to T_{\alpha}$. We claim that for $\alpha \in [0,1], x\in T$,  $(\alpha,x)\mapsto \phi_{\alpha}(x)$ defines a convex structure on $T$.

To prove (C1), let $\alpha, \beta \in [0,1]$. Then $\phi_{\alpha}\circ \phi_{\beta}(\rho(\lambda,a))=\phi_{\alpha}\rho(\beta \lambda,a)=\rho(\alpha \beta \lambda,a))=\phi_{\alpha \beta}(\rho(\lambda,a))$. By uniqueness of the extensions, $\phi_{\alpha}\circ \phi_{\beta}=\phi_{\alpha \beta}$, which proves (C1).

To prove (C2), let $\alpha, \beta \in [0,1]$ be such that $\alpha +\beta \in [0,1]$. Then for every $a\in E$, $\lambda \in [0,1]$, $\phi_{\alpha}(\rho(\lambda,a))+\phi_{\beta}(\rho(\lambda,a))= \rho(\alpha \lambda,a)+\rho(\beta \lambda,a)=\rho(\alpha \lambda +\beta \lambda,a)=\phi_{\alpha +\beta}(\rho(\lambda,a))$, which yields $\phi_{\alpha +\beta}=\phi_{\alpha}+\phi_{\beta}$, which is (C2).

(C3) follows from the fact that $\phi_{\alpha}$ is a morphism for every $\alpha$, and (C4) follows from $\phi_1(\rho(\lambda,a))=\rho(\lambda,a)$ for all $\lambda \in [0,1]$ and every $a\in E$.

In the same way we prove that there is a convex structure on the tensor product of $[0,1]\otimes_{\sigma}E$ in the category of $\sigma$-effect algebras. }
\end{example}

\begin{example}\label{ex:hilbert}\rm{ Let $H$ be a complex Hilbert space and let ${\mathcal E}(H)$ be the effect algebra of Hilbert space effects from Example \ref{ex:effects}.
For $\lambda \in [0,1]$ and $a\in {\mathcal E}(H)$, $(\lambda,a)\mapsto \lambda a$, where $\lambda a$ is the usual scalar multiplication for operators, gives a convex structure on ${\mathcal E}(H)$ and so ${\mathcal E}(H)$ becomes a convex effect algebra.
}
\end{example}

\begin{example}\label{ex:meas} \rm{Let $(\Omega,{\mathcal A})$ be a measurable space and let ${\mathcal E}(\Omega,{\mathcal A})$ be the set of measurable functions on $\Omega$ with values in $[0,1]$. If we define, for $f,g\in {\mathcal E}(\Omega,{\mathcal A})$,  $f\perp g$ iff $f+g\leq 1$ (pointwise), and in this case $f\oplus g(\omega):=f(\omega)+g(\omega), \omega \in \Omega$, we obtain an effect algebra, called effect algebra of \emph{fuzzy events}. Moreover, if we define for $\lambda \in [0,1]$, $(\lambda f)(\omega):=\lambda f(\omega)$ as the usual multiplication, we can see that ${\mathcal E}(\Omega,{\mathcal A})$ is a convex effect algebra. Fuzzy events are basic concepts in fuzzy probability theory \cite{BeBu,Gufuzzy}.}
\end{example}

\begin{example}\label{ex:linear} \rm{The latter above two examples are special cases of the following example. Let $V$ be an ordered real linear space with zero $\theta$ and with a  strict positive cone $K$. Recall that the partial order is defined by $x\leq_K y$ iff $y-x\in K$.  We say that $K$ is \emph{generating} iff $V=K-K$. Let $u\in K$ with $u\neq \theta$ and form the interval $[\theta,u]:=\{ x\in K: x\leq_K u\}$.  For $x,y\in [\theta,u]$ we define $x\perp y$ iff $x+y\leq_K u$ and in this case we define $x\oplus y=x+y$. It is easy to check that $([\theta,u];\oplus, \theta,u)$ is an effect algebra with $x^{\perp}=u-x$ for every $x\in [\theta,u]$.  A straightforward verification shows that $(\lambda,x)\mapsto \lambda x$, $\lambda \in [0,1]$ is a convex structure on $[\theta,u]$ so that $([\theta,u]; \oplus, \theta, u)$ is a convex effect algebra. We say that $[\theta,u]$ \emph{generates} $K$ iff $K={\mathbb R}^+.[\theta,u]$ and we say that $[\theta,u]$ generates $V$ iff $[\theta,u]$ generates $K$ and $K$ generates $V$ \cite{GuPurepr}.}
\end{example}

It has been shown in \cite{GuPurepr} that every convex effect algebra is equivalent to a convex effect algebra described in Example \ref{ex:linear}. We have the following result. We recall that two ordered linear spaces $(V_1;K_1)$ and $(V_2;K_2)$ are \emph{order isomorphic} iff there exists a linear bijection $T:V_1\to V_2$ such that $T(K_1)=K_2$.

\begin{theorem}\label{th:repr}{\rm \cite[Theorem 3.1]{GuPurepr}} Every convex effect algebra $(E; \oplus,0,1)$ is affinely isomorphic to an effect algebra $([\theta,u];\oplus, \theta,u)$, where $[\theta,u]$ is a generating interval for an ordered  linear space $(V;K)$; and the effect algebra order $\leq$ on $[\theta,u]$ coincides with the linear space order $\leq_K$ restricted to $[\theta,u]$. Moreover, $(V;K)$ is unique in the sense that if $E$ is affinely isomorphic to an interval $[\theta_1,u_1]$ that generates $(V_1;K_1)$, then $(V_1;K_1)$ is order isomorphic to $(V;K)$.
\end{theorem}

The following lemma shows that if an effect algebra $E$ is embedded into the tensor product $[0,1]\otimes E$, then
a state operator on $E$ can be extended to a state operator on $[0,1]\otimes E$.

\begin{lemma}\label{le:extensor} Let $\tau:E\to E$ be a state operator on the effect algebra $E$. Let $(T,\rho)$ be the tensor product of $[0,1]$ and $E$. Then $\tau$ uniquely extends to a state operator on $T$.
\end{lemma}

\begin{proof} Define a mapping $\nu:[0,1]\times E\to T$ by $\nu(\lambda,a)=\rho(\lambda, \tau(a))$.
Since $\tau$ is a state operator and $\rho$  a bimorphism, we obtain that $\nu$ is a bimorphism. Therefore, by the universal property of the tensor product, there is a unique (effect algebra) morphism $\tau^o :T\to T$ such that $\nu=\tau^o\circ \rho$.  Moreover, $\tau^o(\tau^o(\rho(\lambda,a))=\tau^o(\rho(\lambda, \tau(a))=\rho(\lambda, \tau(\tau(a)))=\rho(\lambda, \tau(a))=\tau^o(\rho(\lambda,a))$.  Since  $T$ is generated by elements of the form $\rho(\lambda,a)$, we obtain that $\tau^o$ is idempotent.
\end{proof}

\begin{lemma}\label{le:affine}  Let $\tau :E\to E$ be a state operator on a convex effect algebra $E$. Then for every $\alpha \in [0,1]\subseteq {\mathbb R}$ and every $a\in E$, $\tau (\alpha a)=\alpha \tau(a)$.
\end{lemma}

\begin{proof} Let $n$ be a nonnegative  integer and $a\in E$ such that $na$ exists in $E$. Then $\tau(na)=\tau(a\oplus a\oplus \cdots \oplus a)=\tau(a)\oplus \cdots \oplus \tau(a)=n\tau(a)$. Now assume that $n>0$, $a\in E$, then
\begin{eqnarray*}
n\tau(\frac{1}{n}a)&=& \tau(\frac{1}{n}a)\oplus \tau(\frac{1}{n}a)\oplus \cdots \oplus \tau(\frac{1}{n}a)\\
&=& \tau(\frac{1}{n}a\oplus \cdots \oplus \frac{1}{n}a)=\tau(n(\frac{1}{n}a))=\tau(a),
\end{eqnarray*}
hence $\tau(\frac{1}{n}a)=\frac{1}{n}\tau(a)$. It follows that for every $m\leq n$ we have
$\tau(\frac{m}{n}a)=\frac{m}{n}\tau(a)$.

For every $\alpha \in [0,1]$ there are sequences of rational numbers $(q_n)_n, (r_n)_n$ such that $q_n,r_n\in{\mathcal Q}\cap[0,1]$ and $q_n\uparrow \alpha, r_n\downarrow \alpha$. For every $a\in E$, $q_na \uparrow \alpha a$, $r_na\downarrow \alpha a$. Therefore $q_n\tau(a)\leq \tau(\alpha a)\leq r_n\tau(a)$. Taking the limits for $n\to \infty$, we get $\tau(\alpha a)=\alpha \tau(a)$.
\end{proof}

\begin{theorem}\label{th:extend} Let $E$ be a convex effect algebra, and let $[\theta,u]$ be the generating interval in an ordered linear space $(V;K)$ induced by $E$ by Theorem \ref{th:repr}. Then any state operator $\tau:E\to E$ extends to a linear, idempotent endomorphism $p:V\to V$ such that $p(K)\subseteq K$. In addition, if   $x,y\in V^+$, and $\tau$ is strong,  then $p(p(x)\wedge p(y))=p(x)\wedge p(y)$ whenever  $p(x)\wedge p(y)$ exists.
\end{theorem}

\begin{proof} We can (and do) identify $E$ with the interval $[\theta, u]$. Let $a\in V^+$, then since $u$ is an order unit, there is $n\in {\mathbb N}$ such that $a\leq nu$, hence $\frac{1}{n}a\leq u$. Put $p(a):=n\tau(\frac{1}{n}a)$.
Let $a,b\in V^+$, we find $n$ such that $\frac{a+b}{n}\leq u$, and so $p(a+b)=n\tau(\frac{a+b}{n})=n(\tau(\frac{a}{n})\oplus \tau(\frac{b}{n}))=p(a)+p(b)$. Now for every $x\in V$, $x=x_1-x_2$, $x_1,x_2\in V^+$, and we may put $p(x)=p(x_1)-p(x_2)$. If also $x=y_1-y_2$, $y_1,y_2\in V^+$, then
from $x_1+y_2=y_1+x_2$, we get $p(x_1)+p(x_2)=p(y_1)+p(y_2)$, so that $p(x)$ is well defined. For $a\in V^+$, $p(p(a))=p(n\tau(\frac{a}{n}))=np(\tau(\frac{a}{n}))=n\tau(\tau(\frac{a}{n}))=p(a)$. From this we derive that $p^2=p$.

To finish the proof, assume that $\tau$ is strong. Let $p(x)\wedge p(y)$  exist for $x,y\in V^+$. Then there is $n\in \mathbb N$ such that $\theta \leq x,y\leq n 1$, hence $\theta\leq x/n, y/n\leq 1$. Therefore $p(p(x)\wedge p(y))=p(n(\tau(x/n)\wedge \tau(y/n)))=n\tau(\tau(x/n)\wedge \tau(y/n))=n(\tau(x/n)\wedge \tau(y/n))=p(x)\wedge p(y)$.

\end{proof}

\section{State operators and conditional expectations  on JC-effect algebras}

Let us consider the following examples.

\begin{example}\label{ex:2}  {\rm Let $ (\Omega, \Sigma, \mu )$ be a  (finite) measure space. Let ${\mathcal E}$ be the set of all $\Sigma$-measurable functions $f:\Omega \to [0,1]$. ${\mathcal E}$ can be organized into an MV-algebra by putting $f\boxplus g=\min \{f+g,1\}$ and $f'=1-f$, where $1(\omega)=1$ for all $\omega \in \Omega$. Moreover, we also have ${\mathcal E}= \{ f: 0\leq f\leq 1\}\subseteq L^2(\Omega,\Sigma,\mu)$. It has been shown that a linear  transformation $T:L^2 \to L^2$ is a conditional expectation if and only if  (0) $T$ is self-adjoint and idempotent (i.e, $T$ is a projection); (1) $T(1)=1$; and  (2) if $f\in L^2, g\in L^2$ then  $T(\max \{Tf,Tg\})=\max \{Tf,Tg\}$ \cite{Sid}. It can be easily seen that if we restrict $T$ to ${\mathcal E}$, we obtain a strong state operator on ${\mathcal E}$.}

\end{example}

\begin{example}\label{ex:1}{\rm
With the standard quantum-mechanical Hilbert space formalism, $\mu_B(x):=\mu(\sum_{i=1}^{\infty}p_ixp_i)$ , $x\in {\mathcal B}(H)$ is the state after a L\"uders  -- von Neumann measurement of an observable $B=\sum_{i=1}^{\infty}b_ip_i$  with discrete spectrum $(b_i)_{i=1}^{\infty}$ and with the eigen-projections $(p_i)_{i=1}^{\infty}, \sum_{i=1}^{\infty}p_i=I$, when the physical system was in the initial state $\mu$ prior to the measurement. While von Neumann \cite{vN} originally considered only the case when the $p_i$ are one-dimensional projections (i.e. atoms in the projection lattice ${\mathcal P}(H)$ of the Hilbert space $H$), that is, only an observable $B$ with a non-degenerate spectrum,  L\"uders \cite{Lu} later extended von Neumann's measurement process to the case when $p_i$  need  not longer be  one-dimensional, and the observable $B$ may be degenerated.

The operator $p:=\sum_{i=1}^{\infty}p_ixp_i$ is the \emph{von Neumann-L\"uders conditional expectation} on ${\mathcal B}(H)$ with respect to $B$. Let $\tau(a):=\sum_{i\in \mathbb N}p_iap_i$, $a\in {\mathcal E}(H)$, be the restriction of $p$ to ${\mathcal E}(H)$.

It is easy to check that this  $\tau$ satisfies axioms (i), (ii), (iii) of Definition \ref{de:stateop}. Moreover, if $a$ commutes with all $p_i$, then $\tau(a)=a$, and conversely, $\tau(a)=a$ implies $p_ja=p_j\sum_{i=1}^{\infty}p_iap_i=p_jap_j$, whence the range of $\tau$ consists of all effects that commute with all of the projections $p_i$.

Moreover, the state operator induced by the von Neumann-L\"uders conditional expectation is faithful, hence it is strong. Indeed, let $\sum_ip_iap_i=0$, then $p_iap_i=0$, hence $p_ia=0$ for all $i$, and hence $\sum_ip_ia=a=0$.

Since ${\mathcal E}(H)$ is far from being a lattice \cite{Gud, LM}, but it is closed under the triple products of the form $aba$, $a,b\in {\mathcal E}(H)$, it may be reasonable to consider, instead of property (\ref{eq: x}), the  following property:

\begin{equation}\label{eq: xx}
\forall a\in {\mathcal E}(H), \tau(\tau(a)b\tau(a))=\tau(a)\tau(b)\tau(a).
\end{equation}

It is easy to check that the state operator induced by the von Neumann-L\"uders conditional expectation satisfies (\ref{eq: xx}).}
\end{example}

Motivated by the latter examples, we will study state operators on so-called JC-effect algebras, which generalize the effect algebra ${\mathcal E}(H)$ and their relations with conditional expectations.

Recall that a JC-algebra is a norm-closed real vector subspace of bounded self-adjoint operators on a Hilbert space $\mathcal H$, closed under the Jordan product $a\circ b=\tfrac 12(ab+ba)$, \cite{ESjord}. A JC-algebra is called a JW-algebra if it is closed in the weak topology. It was shown that the lattice of projections in a JW-algebra $\jca$ must be complete, hence $\jca$ contains a unit $1$, \cite{Top}. In what follows, we will  suppose that a JC-algebra contains the unit as well. In particular, if $\jca=\mathcal A^{sa}$ is the set of all self-adjoint operators in a C*-subalgebra $\mathcal A\subset B(\mathcal H)$, we will suppose that $\mathcal A$ is unital.

Let $\jca$ be a JC-algebra and let $\jca^+=\jca\cap B(\mathcal H)^+$ be the cone of positive operators in
$\jca$. Then $\jca^+=\{ a\circ a: a\in \jca\}$. The set of effects in $\jca$ is called a \emph{JC-effect algebra}, it will be denoted by
$\mathcal E(\jca)=\jca\cap \mathcal E(\mathcal H)$. A sub-effect algebra $\mathcal E\subset \mathcal E(\jca)$ is called  a sub-JC-effect algebra if $a^2\in\mathcal E$ for all $a\in \mathcal E$.

Recall that the triple product on a Jordan algebra is defined by
\[
\{xyz\}=(x\circ y)\circ z+(y\circ z)\circ x-(z\circ x)\circ y.
\]
In a  JC-algebra $\jca$, we have
\[
\{abc\}=\frac12(abc+cba),\qquad a,b,c\in\jca,
\]
in particular,
\[
\{aba\}=2(a\circ b)\circ a -(a^2)\circ b=aba.
\]
It is clear that if $a,b\in \mathcal E(\jca)$, then $aba\in \mathcal E(\jca)$ and if $\mathcal E\subset \mathcal E(\mathcal H)$ is a sub-effect algebra, then $\mathcal E$ is a sub-JC-effect algebra if and only if $aba\in \mathcal E$ whenever $a,b\in \mathcal E$.

Let $\tau :{\mathcal E}({\jca})\to {\mathcal E}({\jca})$ be a state operator. Since $\mathcal E(\jca)$
 is the $[0,I]$ interval in the ordered vector space $(\jca,\jca^+)$, $\tau$ extends to a linear, positive idempotent and unital mapping $\tilde \tau: \jca\to \jca$, by Theorem \ref{th:extend}.
Such maps were studied in \cite{ES} and it was shown that the image $\tilde \tau(\jca)$
is (has an isometric Jordan representation as) a JC-algebra, with the product $\tilde\tau(a)\star\tilde\tau(b)=\tilde\tau(\tilde\tau(a)\circ\tilde\tau(b))$.  Moreover,  if $\tilde \tau$ is faithful, $\tilde \tau(\jca)$ is a Jordan subalgebra in $\jca$. All of the results in this section are easy consequences of the results of \cite{ES}, nevertheless, we include some proofs for the convenience of the reader.

A crucial result used in this paragraph is the Kadison-Schwarz inequality \cite{K}, extended to the context of JC-algebras \cite{ES},
which states that
 if $p:\jca\to B(\mathcal H)$ is a linear, positive and contractive map, then $p(a^2)\ge p(a)^2$ for all $a\in \jca$.
In particular,  a state operator $\tau$ on $\mathcal E(\jca)$ satisfies
\begin{equation}\label{eq:ineq}
 \tau(a)^2\le \tau( \tau(a)^2)\le \tau(a^2),\qquad a\in \mathcal E(\jca).
\end{equation}
We will show below that unital additive  maps satisfying equality in either of these  inequalities are special cases
 of state operators. Moreover, we will show that each state operator is a composition of such maps.

\begin{lemma}\label{le:equiv}
Let $p:{\jca}\to {\jca}$ be a positive, linear mapping such that  $p(1)=1$ and let $a\in \jca$.
The following are equivalent:

{\rm(a)} $p(a^2)=p(a)^2$;

{\rm(b)} $p(b\circ a)=p(b)\circ p(a)$ for all $b\in \jca$;

{\rm(c)} $p(aba)=p(a)p(b)p(a)$ for all $b\in \jca$.

\end{lemma}

\begin{proof} The first part of the proof is completely analogical to a proof of a similar statement for
 maps on C*-algebras satisfying Schwarz inequality, \cite{ham}.

(a) $\implies$ (b): Let $p(a^2)=p(a)^2$, $b\in \mathcal A^{sa}$, $t\in \mathbb R$. Then by the Kadison-Schwarz inequality,
\begin{align}
2tp(a)\circ p(b)=& p(ta+b)^2-t^2p(a)^2-p(b)^2\notag\\
\le& p((ta+b)^2)-t^2p(a^2)-p(b)^2\label{eq:lem}\\
=& 2tp(a\circ b)+p(b^2)-p(b)^2\notag
\end{align}
Dividing the inequality by $t$ and letting $t\to \pm \infty$, we obtain (b).

(b) $\implies$ (c): Note that since $a^n=a\circ a^{n-1}$, we obtain by induction from (b) that $p(a^n)=p(a)^n$, $n\in \mathbb N$. In particular, $p(a^2)^2=p(a)^4=p((a^2)^2)$, so that both $a$ and $a^2$ satisfy the condition (a).
 Further, we have
\[
p((a\circ b)\circ a)=p(a\circ b)\circ p(a)=(p(a)\circ p(b))\circ p(a)
\]
and by the part (a) $\implies$ (b) of the proof,
\[
p(a^2\circ b)=p(a^2)\circ p(b)=p(a)^2\circ p(b).
\]
Hence
\[
p(aba)=p(2(a\circ b)\circ a -a^2\circ b)=2(p(a)\circ p(b))\circ p(a)-p(a)^2\circ p(b)=p(a)p(b)p(a).
\]

(c) $\implies$ (a) is obvious, by putting $b=1$.

 \end{proof}

\begin{definition}\label{de:condexp} A mapping $\tau:{\mathcal E}(\jca)\to {\mathcal E}(\jca)$ will be called a
\emph {conditional expectation} on ${\mathcal E}(\jca)$ iff the following conditions are satisfied:
\begin{enumerate}
\item[{\rm(i)}] $\tau(1)=1$;
\item[{\rm(ii)}] if $a\perp b$, then $\tau(a+b)=\tau(a)+\tau(b)$;
\item[{\rm(iii)}] for all $a,b \in {\mathcal E}(\jca)$, $\tau(\tau(a)b\tau(a))=\tau(a)\tau(b)\tau(a)$.
\end{enumerate}
\end{definition}

Notice that putting $b=1$ in (iii) we get
\begin{enumerate}
\item[(iii')] $\tau(\tau(a)^2)=\tau(a)^2$,
\end{enumerate}
that is, equality holds in the first inequality of (\ref{eq:ineq}). We show in Corollary \ref{co:equiv} that we may
 replace (iii) with (iii') in Definition \ref{de:condexp}.

We will next study the relation between state operators and conditional expectations on $\mathcal E(\jca)$.

\begin{theorem}\label{le:condsa} Let $\tau: \mathcal E(\jca)\to \mathcal E(\jca)$ be a conditional expectation, then $\tau$ is a state
operator on $\mathcal E(\jca)$. Its extension $\tilde \tau$ satisfies
\[
\tilde\tau(\{\tilde \tau(a)b\tilde\tau(c)\})= \{\tilde\tau(a)\tilde\tau(b)\tilde\tau(c)\},\qquad a,b,c\in \jca
\]
In particular, the image $\tilde{\tau}(\jca)$  is a Jordan subalgebra of $\jca$.
\end{theorem}

\begin{proof}  To show that $\tau$ is a state operator, it is enough to check that $\tau^2=\tau$. By property (iii) of
conditional expectations we obtain $\tau(\tau(a)^2)=\tau(a)^2$ for all $a\in {\mathcal E}(\jca)$.
Since  $\tilde\tau(a)=\|a\|\tau(\frac a{\|a\|})$ for positive elements in $\jca$ and $\|a\|1- a\ge 0$ for all $a\in \jca$,
 it is easily checked that we have
\begin{equation}\label{eq:square}
\tilde\tau(\tilde\tau(a)^2)=
\tilde\tau(a)^2,\qquad a\in \jca
\end{equation}
 Applying this equality  to $1+a$ for $a\in {\mathcal E}(\jca)$, we obtain
\[
2\tau(a)=\tilde\tau(1+a)^2-\tau(a)^2-1=\tilde\tau(\tilde\tau(1+a)^2)-\tau(\tau(a)^2)-1=2\tau(\tau(a)),
\]
hence $\tau$ is idempotent.

It follows that $\tilde \tau(a)$ satisfies the property (a) of Lemma \ref{le:equiv} for all
  $a\in \jca$. The rest of the proof now follows by Lemma \ref{le:equiv} (c) by linearity of the triple product.

\end{proof}
\begin{corollary}\label{co:equiv}  Let $\tau:{\mathcal E}(\jca)\to {\mathcal E}(\jca)$ be a unital additive mapping.
Then $\tau$ is a conditional expectation
 if and only if its range $\tau(\mathcal E(\jca))$ is a sub-JC-effect algebra in $\mathcal E(\jca)$.
\end{corollary}

\begin{proof} It is clear that the range is a sub-JC-effect algebra iff (iii') holds. Suppose it is true,
 then by the proof of Theorem \ref{le:condsa}, $\tau$ is idempotent, so that the condition
(iii) of Definition \ref{de:condexp} follows by Lemma \ref{le:equiv}. The converse statement is obvious.

\end{proof}

\begin{theorem}\label{th:faithful} Let $\tau: \mathcal E(\jca)\to \mathcal E(\jca)$ be a faithful state operator. Then $\tau$ is a conditional
expectation on $\mathcal E(\jca)$.
\end{theorem}

\begin{proof}
Let $a\in \mathcal E(\jca)$, then by the Kadison-Schwaz inequality, $e:=\tau(\tau(a)^2)-\tau(a)^2$ is an element in $\mathcal E(\jca)$ such that
$\tau(e)=0$. Since $\tau$ is faithful, this implies that $\tau(\tau(a)^2)=\tau(a)^2$. The rest now follows by Lemma \ref{le:equiv}.

\end{proof}

We now consider equality in the second inequality of (\ref{eq:ineq}).

\begin{lemma}\label{le:second}
Let $p:\jca\to\jca$ be a linear positive unital idempotent map. Let $a\in \jca$. Then the following are
 equivalent.
 \begin{enumerate}
\item[(a)] $p(a^2)=p(p(a)^2))$;
\item[(b)] $p(a\circ b)=p(p(a)\circ p(b))$, for all $b\in \jca$;
\item[(c)] $p(aba)=p(p(a)p(b)p(a))$, for all $b\in \jca$.

 \end{enumerate}

\end{lemma}

\begin{proof} The proof is very similar to the proof of Lemma \ref{le:equiv}.

\end{proof}

Note that if $a=p(a)$ is in the range of $p$ in the above lemma, then $a$ trivially satisfies condition (a), hence we have
\begin{align}
p(p(a)\circ b)&=p(p(a)\circ p(b))\label{eq:ES}\\
p(p(a)bp(a))&= p(p(a)p(b)p(a))
\end{align}
for all $a,b\in \jca$.

\begin{theorem}\label{thm:jordan} Let $\tau:\mathcal E(\jca)\to \mathcal E(\jca)$ be a unital additive map such that $\tau(a^2)=\tau(\tau(a)^2)$ for
 all $a\in \mathcal E(\jca)$. Then $\tau$ is a state operator. Moreover,
 \[
\mathcal I_\tau:=\{a\in \jca,\ \tilde \tau(a^2)=0\}
 \]
is a Jordan  ideal in $\jca$ and the map $[a]_{\mathcal I_\tau}\mapsto \tilde \tau(a)$ is an isometric
Jordan isomorphism of $\jca|_{\mathcal I_\tau}$ onto the
range $\tilde \tau(\jca)$ with Jordan structure given by $a\star b=\tilde \tau(a\circ b)$.
\end{theorem}

\begin{proof} It is clear that $\tilde \tau(a^2)=\tilde\tau(\tilde\tau(a)^2)$ extends to  all $a\in \jca^+$, in particular to $a+I$, $a\in \mathcal E(\jca)$. It follows that $\tau$ is idempotent, hence a state operator. Using Lemma \ref{le:second}, it is easy to see that $\tilde \tau(a^2)=\tilde\tau(\tilde\tau(a)^2)$ holds for all $a\in\jca$.

We next show that $\mathcal I_\tau=\{ c\in \jca, \tilde\tau(c)=0\}$.
Indeed, let $c\in \mathcal I_\tau$, then by Kadison inequality, $\tilde \tau(c)^2\le \tilde\tau(c^2)=0$.
Conversely, if $\tilde \tau(c)=0$, then $\tilde\tau(c^2)=\tilde\tau(\tilde\tau(c)^2)=0$. Hence $\mathcal I_\tau$
is a Jordan ideal  if and only if $bab\in \mathcal I_\tau$  for all
$a\in \mathcal I_\tau$ and $b\in \jca$, \cite{ESjord}. By Lemma \ref{le:second},
\[
\tilde\tau(bab)=\tilde\tau(\tilde\tau(b)\tilde\tau(a)\tilde\tau(b))=0
\]

By  \cite{ESjord}, $\jca|_{\mathcal I_\tau}$ is a JC-algebra.
It is clear that
$\phi: [a]\mapsto \tilde \tau(a)$ is a well defined linear  unital map $\jca|_{\mathcal I_\tau}$ onto
$\tilde \tau(\jca)$. Moreover,
\[
\phi([a]^2)=\phi([a^2])=\tilde\tau(a^2)=\tilde\tau(\tilde \tau(a)^2), \qquad [a]\in \jca|_{\mathcal I_\tau}
\]
this implies that $\phi$ is a Jordan homomorphism with respect to the product $a\star b$
 on $\tilde\tau(\jca)$.
To show that $\phi$ is an isometry, note that $a-\tilde\tau(a)\in \mathcal I_\tau$ for all $a$ and hence
\[
\|[a]\|=\inf \{\|a+c\|, c\in \mathcal I_\tau\}\le \|\tilde\tau(a)\|=\|\tilde\tau(a+c)\|\le \|a+c\|,\quad c\in \mathcal I_\tau.
\]

\end{proof}

\begin{definition} A map on $\mathcal E(\jca)$ satisfying the conditions of Theorem \ref{thm:jordan} will be called
a Jordan state operator on $\mathcal E(\jca)$.
\end{definition}

We now turn to the case of a JW-algebra.

\begin{definition}\label{de:normal} We say that an additive map $\tau:{\mathcal E}({\jca})\to {\mathcal E}({\jca})$
is \emph{normal} iff for every ascending net $(a_{\alpha})_{\alpha}$,
$a_{\alpha}\nearrow a$ implies $\tau(a_{\alpha}) \nearrow \tau(a)$.
\end{definition}

Notice that $\tau$ is normal iff for any summable family  $(a_i)_{i\in I}$ of effects, we have $\tau(\oplus_{i\in I}a_i)=\oplus_{i\in I}\tau(a_i)$.
Recall that a state $\rho$ on an effect algebra is \emph{completely additive} iff $\rho(\oplus_{i\in I}a_i)=\sum_{i\in I}\rho(a_i)$ whenever the orthosum $\oplus_{i\in I}a_i$ exists.

Let ${\jca}$ be a JW-algebra, then ${\mathcal E}({\jca})$ is an orthocomplete effect algebra, and completely additive
states on ${\mathcal E}({\jca})$ coincide with the restrictions of normal states on ${\jca}$ to ${\mathcal E}({\jca})$.
Hence
 a  map $\tau$  on $\mathcal E(\jca)$ is normal if and only if $\phi\circ \tau$ is normal for all normal states $\phi$.

\begin{corollary}
Let $\jca$ be a JW-algebra and let $\tau:\mathcal E(\jca)\to\mathcal E(\jca)$ be a unital normal map. Then
 $\tau$ is a conditional expectation if and only if its range is a sub-JW-effect algebra in $\mathcal E(\jca)$.

\end{corollary}

The following definition is analogous to the definition of a support of a normal state \cite[Definition 7.2.4]{KaRi}.

\begin{definition}\label{de:support}  For a normal state operator $\tau$ on ${\mathcal E}({\jca})$,  the \emph{support} of $\tau$ is the complement $e$ of the maximal projection $f\in {\mathcal E}({\jca})$ with the property $\tau(f)=0$.
\end{definition}

By Lemma 1.2 in \cite{ES}, we have that
  $\tau(a)=\tau(eae)$ for every $a\in {\mathcal E}({\jca})$ and $\tau(a)=0$ implies $eae=0$, moreover, $e\tau(a)=\tau(a)e$
 for every $a\in \mathcal E(\jca)$.

\begin{corollary}\label{co:support} Let $\tau$ be a normal state operator on ${\mathcal E}({\jca})$, and let $e$ be the support of $\tau$. Then $\tau_e$,
defined by $\tau_e(a)=\tau(a)e$, is a normal  faithful conditional expectation on $e{\mathcal E}({\jca})e ={\mathcal E}({\jca})[0,e]$.
\end{corollary}

Let us now turn to the case when the  state operator $\tau$ on a JC-effect algebra $\mathcal E(\jca)$ is not faithful. The next Theorem (together with Theorem \ref{thm:jordan}) is a reformulation of \cite[Corollary 1.5]{ES}).

\begin{theorem}\label{th:dec}
Let $\jca$ be a JW-algebra and let $\tau: \mathcal E(\jca)\to \mathcal E(\jca)$ be a normal state operator.
Then
\begin{enumerate}
\item[(i)] $\mathcal E_\tau:=\{a\in \mathcal E(\jca), \tau(a^2)=\tau(\tau(a)^2)\}$ is a sub-JW-effect algebra in $\mathcal E(\jca)$;
\item[(ii)] there is a faithful normal conditional expectation $\mu$ on $\mathcal E(\jca)$ with range $\mathcal E_\tau$;
\item[(iii)] there is a normal Jordan state operator
$\phi$ on $\mathcal E_\tau$ such that $\tau=\phi\circ \mu$.
\end{enumerate}

\end{theorem}

\begin{proof} Let $e$ be the support of $\tau$.
Let us define
\[
\mu(a):=\tau(a)e+(1-e)a(1-e)=\tau_e(eae)+(1-e)a(1-e), \qquad a\in \mathcal E(\jca).
\]
It is easy to see  that $\mu$ is a normal state operator  on $\mathcal E(\jca)$.
 Suppose that $\mu(a)=0$ for some $a\in \mathcal E(\jca)$, then we must have
$\tau(a)e=(1-e)a(1-e)=0$. It follows that  $\tau(a)=\tau(\tau(a))=\tau(\tau(a)e)=0$, so that $eae=ae=0$
 and $a=ae+a(1-e)=0$, hence $\mu$ is faithful. By Theorem \ref{th:faithful}, $\mu$ is a conditional expectation.

We will show that the range of $\mu$ is $\mathcal E_\tau$: let $a=\mu(a)$, then it is quite clear that $\tau(a^2)=\tau(\tau(a)^2)$.
Conversely, suppose $a\in \mathcal E_\tau$, then by using (\ref{eq:ES}), $\tilde \tau((a-\tau(a))^2)=0$, hence $(a-\tau(a))e=e(a-\tau(a))=0$. It follows that $ae=ea=\tau(a)e$ and $\mu(a)=ae+a(1-e)=a$.
This proves (i) and (ii).

Let $\phi$ be the restriction of $\tau$ to  $\mathcal E_\tau$, then $\phi$ is clearly a normal Jordan state operator and
 we have
\[
\phi\circ\mu(a)=\tau(\tau(a)e+(1-e)a(1-e))=\tau(a),\qquad a\in \mathcal E(\jca),
\]
so that (iii) is true.

\end{proof}

Let $\jca$ be a JC-algebra and let $\mathcal A$ be the C*-subalgebra in $B(\mathcal H)$ generated by $\jca$. Then we may identify
 the second dual $\jca^{**}$ with the strong operator closure of $\jca$ in the second dual $\mathcal A^{**}$ of $\mathcal A$,
 $\jca^{**}$  is thus a JW-algebra, \cite{ESjord,ES}. Moreover, if $p$ is a unital positive projection on $\jca$, then $p$
  extends to a normal unital positive projection $p^{**}:\jca^{**}\to \jca^{**}$.

  Let now $a$ be any element in the JW-effect algebra $\mathcal E(\jca^{**})$, then since the Kaplansky density theorem holds   for JC-algebras,  there is a net $\{a_\alpha\}$ of elements in the unit ball of $\jca$ converging to $a^{1/2}$ in the strong operator topology.
It follows that $\{a_\alpha^2\}$ is a net of element in $\mathcal E(\jca)$, converging to $a$, so that $\mathcal E(\jca^{**})$
 is the strong operator closure of $\mathcal E(\jca)$ in $\jca^{**}$.
It also folows that any state operator $\tau$ on $\mathcal E(\jca)$
   extends to a normal state operator $\tau^{**}$ on $\mathcal E(\jca^{**})$, given  by the restriction of $\tilde\tau^{**}$.

\begin{corollary} Let $\jca$ be a JC algebra and let $\tau$ be a state operator on $\mathcal E(\jca)$.
Then
 there is a faithful normal conditional expectation $\mu$ on $\mathcal E(\jca^{**})$ and a  Jordan state operator
$\phi$ on the range of $\mu$ such that $\tau=\phi\circ \mu|_{\mathcal E(\jca)}$.

\end{corollary}

\section{State operators and conditional expectations on effect algebras of Abelian C*-algebras}\label{sec:jca}

Let ${\mathcal A}$ be an abelian C*-algebra. Then ${\mathcal A}$ is isomorphic with $C(X)$, the set of all continuous complex valued functions on a compact Hausdorff space $X$.
Notice that the unit interval ${\mathcal E}({\mathcal A})$  is the MV-effect algebra consisting of all continuous functions $f:X\to [0,1]\subseteq {\mathbb R}$, this will be denoted by $C_1(X)$.
In fact, $C_1(X)$ is the unit interval in the real abelian C*-algebra $C(X;\mathbb R)$ of continuous real functions on $X$, which is a JC-algebra, with Jordan product being the usual product of functions. Hence $C_1(X)$ is also a JC-effect algebra. Hence we may compare the notion of a strong state operator and a conditional expectation.

\begin{theorem}\label{th:str_cond} Let $X$ be a compact Hausdorff space and let
$\tau$ be a state operator on $C_1(X)$. Then $\tau$ is a conditional expectation if and only if $\tau$ is strong.
\end{theorem}

\begin{proof} Let $\tau$ be a conditional expectation on $M$. By Corollary \ref{co:equiv}, the range of $\tau$ is a JC-effect-subalgebra of $M$. It follows that $\tau(M)$ is the unit interval in a C*-subalgebra of $C(X)$, and is therefore lattice ordered.
This implies that $\tau$ is strong.

Let $\tau$ be a strong state operator on $C_1(X)$ and let $\tta$ be its extension to $C(X;\mathbb R)$.
 The range $\tta(C(X; {\mathbb R}))$ is a closed linear subspace of $C(X;{\mathbb R})$ with the positive cone
$\tta(C(X,{\mathbb R}))\cap C(X;{\mathbb R})^+=\tta(C(X;\mathbb R)^+)$. By Theorem \ref{th:extend},  we obtain that
$\tta(C(X;{\mathbb R})^+)$ is lattice ordered, and hence $\tta(C(X;{\mathbb R}))$ is a lattice.
Let us define a relation $x\sim y$, $x,y\in X$, iff $f(x)=f(y)$ for all $f\in \tta(C(X;{\mathbb R}))$. Then $\sim$ is an equivalence, and let $Y:=\{ [x]: x\in X\}$ denote the quotient space, then $Y$ is compact Hausdorff.
For all $f\in \tta(C(X;{\mathbb R}))$, define $\tilde{f}[x]=f(x)$, then $\tilde{f}[x]$ is well defined, and does not depend on the choice of $x\in [x]$.
The mapping $f\mapsto \tilde{f}$ is an isometric linear isomorphism $C(X;\mathbb R)\to C(Y;\mathbb R)$, and $({f\wedge g})^{\sim}[x]=f\wedge g(x)=f(x)\wedge g(x)=\tilde{f}[x]\wedge \tilde{g}[x]$. The set of functions
$B:=\{ \tilde{f}:Y\to {\mathbb R}: f\in \tta(C(X;{\mathbb R}))\}$ satisfies the following conditions:

\begin{enumerate}
\item[(i)]  $B$ separates points.
\item[(ii)]   $B$ contains the constant function $1$.
\item[(iii)]  If $f\in B$ then $\alpha f\in B$ for all constants $\alpha \in {\mathbb R}$.
\item[(iv)]  $B$ is a boolean ring; that is, if $f, g \in B$, then $f+g \in B$ and $\max \{f,g\} \in B$.
\end{enumerate}
By the boolean ring version of the Stone - Weierstrass theorem  \cite[Theorem 7.29]{HeSt},
$B$ is dense in $C(Y,{\mathbb R})$ hence $B=C(Y;{\mathbb R})$ and  the range $\tta(C(X;{\mathbb R}))$ of $\tta$ is a
JC- subalgebra of $C(X;{\mathbb R})$. By Corollary \ref{co:equiv}, $\tau$ is a conditional expectation.

\end{proof}

Let $\tau$ be any state operator on $C_1(X)$ and let $\tta$ be its extension to $C(X;\mathbb R)$.
 By decomposing any $f\in C(X)$ into real and imaginary parts, $\tau$ can be extended to $C(X)$,
we  denote this extension  again by $\tta$. It is easy to see that $\tta$ is a  linear positive and idempotent map on
$C(X)$. By Theorems \ref{th:str_cond} and \ref{le:condsa}, $\tau$ is strong if and only if
 $\tta$ is a conditional expectation in the sense of operator algebras, as introduced in
\cite{U}.

Let now  $\tau$ be any state operator on $C_1(X)$ with extension $\tilde{\tau}$ on $C(X)$.
Then
$I_{\tau}:=\{ a\in C_1(X): \tau(a)=0\}$ is an ideal in $C_1(X)$  and $\tau$ induces a faithful state operator on the quotient
$C_1(X)|I_{\tau}$, see Theorem \ref{th:quotfaithful}.  Put $\mathcal I_\tau:=\{ f\in C(X): \tilde\tau(f^*f)=0\}$.

\begin{theorem}\label{th:id-id} Let $\tau$ be a  state operator  on $C_1(X)$. Then  $\mathcal I_\tau$ is a closed ideal in $C(X)$ and $C_1(X)|{I_\tau}\simeq \mathcal E(C(X)|{\mathcal I_\tau})$.
\end{theorem}

\begin{proof}
We start with a simple observation. For all $f\geq 0$, $\tilde{\tau}(f)=0 \, \Leftrightarrow\, \tilde{\tau}(f^2)=0$.
Indeed, let $f\in C(X)$, $f\geq 0$. Assume $\tilde{\tau}(f)=0$. Then $f^2\leq ||f||.f$, hence $0\leq \tilde{\tau}(f^2)\leq ||f||\tilde{\tau}(f)=0$.
Conversely, let $\tilde{\tau}(f^2)=0$. By Kadison inequality, $0=\tilde{\tau}(f^2)\geq \tilde{\tau}(f)^2$, which entails $\tilde{\tau}(f)=0$.

To prove that $\mathcal I_\tau$ is an ideal, let $f,g\in \mathcal I_\tau$ and $\alpha, \beta \in {\mathbb C}$.
Then  $0\leq |\alpha f+\beta g|\leq |\alpha||f|+|\beta||g|$, whence $\tta(|\alpha f+\beta g|)\leq |\alpha|\tta(|f|)+|\beta|\tta(|g|)=0$. This entails  $\tta(|\alpha f+\beta g|^2)=0$, hence $\alpha f+\beta g\in \mathcal I_\tau$. Now let
$f\in \mathcal I_\tau$, $g\in C(X)$. Then
$(fg)^*(fg)=f^*fg^*g\leq |g^*g|f^*f$, whence $\tta((fg)^*(fg))=0$, and $fg\in \mathcal I_\tau$.

Let $f\in C_1(X)$ and let $[f]\in C_1(X)|{I_\tau}$ and $[f]^{\sim}\in C(X)|{\mathcal I_\tau}$
be the equivalence classes containing $f$, i.e. $[f]=\{g\in C_1(X), f\Delta g \in I_\tau\}$ and $[f]^{\sim}=\{g\in C(X),  f-g\in \mathcal I_\tau\}$. For $g\in C_1(X)$, $f\Delta g=|f-g|$, so that
$g\in [f]$ iff $g\in [f]^{\sim}$. Moreover, let $f^{1/2}=h$, then $[f]^{\sim}=[h^2]^{\sim}=([h]^\sim)^2\ge 0$ and similarly for $[1]^{\sim}-[f]^{\sim}=[1-f]^{\sim}$, so that $[f]^{\sim}\in \mathcal E(C(X)|{\mathcal I_\tau})$.
It follows that $[f]\mapsto [f]^{\sim}$ is a well defined injective map $C_1(X)|{I_\tau}\to \mathcal E(C(X)|{\mathcal I_\tau})$ and it is easy to see that it is additive. It is now enough to check that any equivalence class  $\mathcal E(C(X)|{\mathcal I_\tau})$ contains an element of $C_1(X)$.
So let $h\in C(X)$, $[h]^\sim\in \mathcal E(C(X)|{\mathcal I_\tau})$. We may clearly suppose that $h\ge 0$,
 so that $\bar h=h\wedge 1\in C_1(X)$. We will show that $\bar h\in [h]^\sim$.

Since $\mathcal I_\tau$ is a closed ideal in $C(X)$, $\mathcal I_\tau=\{ f\in C(X), f(x)=0, x\in K\}$, where $K\subset X$ is a closed subset given by $K=\{x\in X, f(x)=0, f\in \mathcal I_\tau\}$ \cite[Theorem 3.4.1]{KaRi}.  Since $[1-h]^\sim\ge 0$, there is a positive element $g\in C(X)$
 and some $f\in \mathcal I_\tau$  such that $1-h=g+f$. Hence $h(x)\le 1$ for $x\in K$, so that $h(x)-\bar h(x)=0$ for
 $x\in K$, that is, $h-\bar h\in \mathcal I_\tau$ and $\bar h\in [h]^\sim$.

\end{proof}

We have $C(X)|{\mathcal I_\tau}\simeq C(K)$, where the isomorphism is given by
$[f]^\sim\mapsto f|_K$. We will call the set $K$ the \emph{support of $\tau$}.

\begin{corollary} Let  $\tau$ be a state operator on $C_1(X)$. Then there is a closed subset
$K\subset X$ such that $C_1(X)|{I_\tau}\simeq C_1(K)$. Moreover, $\tau$ induces a faithful conditional expectation
 $\mu$ on $C(K)$, such that $\mu(f|_K)=\tta(f)|_K$, $f\in C(X)$.

\end{corollary}

Let now $\tau$ be a state operator on $M$ and let $\tta$ be its extension to $C(X)$. Then for $x\in X$,
the map $f\mapsto \tta(f)(x)$ is a state on $C(X)$.
 By the Riesz representation theorem, there is some Radon probability measure $\lambda_x$ on $X$ such that
$\tta(f)(x)=\int f(y)\lambda_x(dy)$ for all $f\in C(X)$. Let $K$ be the support of $\tau$.
Since for $f\ge 0$, $f\in \mathcal I_\tau$ iff $\tta(f)=0$, it is easy to see that for any $x\in X$, the support
 of $\lambda_x$ must lie in $K$. Moreover, since any function  $g\in C(K)$ has a continuous extension to $X$,
we see that $g\mapsto \int g(y) \lambda_x(dy)$ defines a linear positive unital map $\varphi: C(K)\to C(X)$
 and $\tta(f)=\varphi(f|_K)$, $f\in C(X)$.

\begin{theorem} A map $\tau$ on $C_1(X)$ is a Jordan state operator if and only if there is a closed
subset $K\subset X$ and a  linear positive unital extension $\varphi:C(K)\to C(X)$ such that
$\tau(f)=\varphi(f|_K)$ for all $f\in C_1(X)$.

\end{theorem}

\begin{proof} Let $\tau$ be a state operator.  Then  $\tau$ is a Jordan state operator if and only if
$\tau(f)(x)=f(x)$ for all $x\in K$, $f\in C_1(X)$ Indeed, if $\tau$ is Jordan, then by the proof of Theorem \ref{thm:jordan}, $g\in \mathcal I_\tau$ iff $\tta(g)=0$, for any real $g\in C(X)$. Hence for $f\ge 0$,
$f^2-\tau(f)^2\in \mathcal I_\tau$, so that $(f^2-\tau(f)^2)(x)=0$ for all $x\in K$. It follows that $f(x)=\tau(f)(x)$
 on $K$ for $f\in C_1(X)$. Conversely, for $x\in K$, $\tau(f)(x)=f(x)$ for all $f\in C_1(X)$ implies
$\tau(f)^2(x)=f^2(x)=\tau(f^2)(x)$, so that $g:=\tau(f^2)-\tau(f)^2\in \mathcal I_\tau$. Since $g\ge 0$, this
 is equivalent with $\tau(g)=0$. Hence $\tau$ is Jordan.

Let now $\tau$ be a Jordan state operator and let $K$ be its support. By the remarks  above,
$\tau$ defines a linear positive unital map $\varphi: C(K)\to C(X)$ such that $\tta(f)=\varphi(f|_K)$.
Moreover, by the previous paragraph, we clearly have $\tta(f)(x)=f(x)$ for $x\in K$ and for all $f\in C(X)$.
 Let $g\in C(K)$ and let $f\in C(X)$ be a continuous extension of $g$, then $\varphi(g)(x)=\tta(f)(x)=f(x)=g(x)$ for
 $x\in K$.

Conversely, if $\tau$ is of the given form, then $\tau$ is clearly a state operator and $\tau(f)(x)=\varphi(f|_K)(x)=
f(x)$ for $x\in K$, so that $\tau$ is Jordan.

\end{proof}

\section{State operators and conditional expectations on convex MV-algebras}

In this section, we will need some facts from the theory of MV-algebras, see \cite{CDM, DvPu} for more details.

Let $(M;\boxplus,^*,0)$ be an MV-algebra, then $M$ is called a $\sigma$- MV-algebra if it is a $\sigma$-lattice.
The set ${\mathcal B}(M):=\{ a\in M:a\boxplus a=a\}$ of all idempotents in $M$, with the operations inherited from $M$, is a boolean $\sigma$-algebra. It is the largest boolean $\sigma$-subalgebra of $M$.

A special example of a $\sigma$-MV-algebra is the following. Let $X$ be a nonempty set. A \emph{tribe} ${\mathcal T}$ over $X$ is a collection of functions ${\mathcal T}\subseteq [0,1]^X$ such that the zero function $0(x)=0$ is in ${\mathcal T}$ and the following is satisfied:
\begin{enumerate}
\item $f\in {\mathcal T}\, \implies \, 1-f\in {\mathcal T}$;
\item $f,g\in {\mathcal T}\,\implies \, f\boxplus g:=\min(f+g,1)\in {\mathcal T}$;
\item $f_n\in {\mathcal T}, n\in \mathbb N$ and $f_n\nearrow f$ (pointwise) $\implies \, f\in {\mathcal T}$.
\end{enumerate}
Every tribe is an MV-algebra - so-called Bold algebra of fuzzy sets, and also a $\sigma$-MV-algebra where the lattice operations $\vee,\wedge$ coincide with the pointwise supremum and infimum, respectively, of $[0,1]$-valued functions defined on $X$. Idempotents in $\mathcal T$ are elements of the form $\chi_B\in \mathcal T$ and
 the sets $B\subseteq X$, $\chi_B\in \mathcal T$
(so-called crisp sets) form  a $\sigma$-algebra of sets, which we also denote by $\mathcal B(\mathcal T)$.  A $\sigma$-additive state $m$ on a tribe ${\mathcal T}$ over $X$ determines a probability measure on ${\mathcal B}({\mathcal T})$, by $P(A):=m(\chi_A)$ for any $A\in {\mathcal B}({\mathcal T})$. By \cite{BuKl}, each $\sigma$-additive state on ${\mathcal T}$ has the following integral representation: for any $f\in {\mathcal T}$, $m(f)=\int_XfdP$.

The following theorem is a generalization of the Loomis-Sikorski theorem to $\sigma$-MV-algebras \cite{MuLS, DvLS}.

\begin{theorem}\label{th:LS}
Let $M$ be a $\sigma$-MV-algebra. Then there exists
 a tribe $M^*$ over a compact Hausdorff space $X$ and a $\sigma$-homomorphism $\eta$ of $M^*$ onto $M$.
\end{theorem}

Elements of the set $X$ can be identified with extremal states on $M$. For each $a\in M$ there exists a unique continuous
function $a^*\in M^*$ such that $\eta(a^*)=a$, this function is given by $a^*(\omega):=\omega(a)$, $\omega\in X$.
A function $f\in M^*$ has the same image as $a^*$ iff the set
$\{ \omega \in X: f(\omega)\neq a^*(\omega)\}$ is a meager set (i.e., a countable union of nowhere dense sets).
Moreover, ${\mathcal B}(M^*)$  is mapped onto ${\mathcal B}(M)$.

Let $m$ be a  $\sigma$-additive state on $M$. Then
$m^*=m\circ \eta$ is a state on $M^*$ and $P^*:A\in {\mathcal B}(M^*)\mapsto m^*(\chi_A)$ is a probability measure on ${\mathcal B}(M^*)$.  For any $a\in M$ we have
\[
m(a)=m(\eta(a^*))=m^*(a^*)=\int_{X} a^*dP^*.
\]
In what follows, by a state we will always mean a $\sigma$-additive state.

The Loomis-Sikorski theorem enables us to  extend some results from probability theory to $\sigma$-MV-algebras.  Conditional expectations on $\sigma$-MV-algebras  have been studied  in \cite{DvPucond}. Similar results for the special case of MV-algebras with product were obtained in \cite{Kr}. Let $M$ be a $\sigma$-MV-algebra, $N$ a sub-$\sigma$-MV-algebra of $M$
and let $N^*$ be the sub-tribe of $M^*$ generated by $\{ b^*\in M^*:b\in N\}$.
Let $m$ be a state on $M$.
 For $b^*\in M^*$ we denote $P^*_{b^*}(a^*):=m^*(b^*\wedge a^*), a^*\in M^*$. Clearly, $P^*_{b^*}|{\mathcal B}(M^*)$ is a $\sigma$-additive measure on ${\mathcal B}(M^*)$, and $P^*_{b^*}(a^*)=P^*(B\cap A)$ if  $a^*=\chi_A, b^*=\chi_B$. The following definition was introduced in \cite{DvPucond}.

\begin{definition}\label{de:Kcond}\textrm{\cite[Definition 4.1]{DvPucond}}  An
MV-conditional expectation of $a\in M$ given $N$ in the state $m$ is a ${\mathcal B}(N^*)$-measurable function
$m(a|{\mathcal B}(N)):X \to {\mathbb R}$
such that for any $b\in {\mathcal B}(N)$,
\begin{equation}\label{eq:Kcond}
\int_{X}m(a|{\mathcal B}(N))(\omega)dP^*_{b^*}(\omega)=m(a\wedge b).
\end{equation}
\end{definition}

Clearly, the function $m(a|{\mathcal B}(N))$ is integrable with respect to $P^*$ and is determined uniquely a.e. $[P^*]$. Definition \ref{de:Kcond} implies that $m(a|{\mathcal B}(N))$ coincides with the Kolmogorovian conditional expectation of $a^*$ given ${\mathcal B}(N^*)$ with respect to $P^*$. Notice that the function $m(a|{\mathcal B}(N))$ need not belong to $N^*$, in general.

In \cite{DvPucond}, the equality  (\ref{eq:Kcond}) was extended to
\begin{equation}\label{eq:mvcond}
\int_{X}m(a|\mathcal B(N))(\omega)dP^*_{b^*}(\omega)=\int_{X}a^*dP^*_{b^*}(\omega)
\end{equation}
for all $b\in N$.
 Moreover, the following was proved.

\begin{theorem}\label{th:4.2}\textrm{\cite[Theorem 4.2]{DvPucond}}  The following properties are satisfied a.e. $[P^*]$:
\begin{enumerate}
\item $m(0|{\mathcal B}(N))=0$, $m(1|{\mathcal B}(N))=1$
\item $a\boxdot b=0$ implies $m(a\boxplus b|{\mathcal B}(N))=m(a|{\mathcal B}(N))+m(b|{\mathcal B}(N))$
\item $\forall a\in M$, $0\leq m(a|{\mathcal B}(N))\leq 1$.
\item For a given sequence $(a_n)_n, a_n\nearrow a$ implies $m(a_n|{\mathcal B}(N))\nearrow m(a|{\mathcal B}(N))$.
\end{enumerate}
\end{theorem}

From now on, we will assume that the $\sigma$-MV-algebra $M$ is a convex MV-effect algebra. It is then easy to see that $M$ is \emph{weakly divisible}, that is,  for any integer $n\in {\mathbb N}$, there is an element $v\in M$ such that $nv$ is defined and $nv=1$; we then  write $v:=\frac{1}{n}$.
  Let $X$ be as in Theorem \ref{th:LS}. Then we have the following.

\begin{theorem}\label{th:wd}\textrm{\cite[Theorem 7.3.12]{DvPu}} An MV-algebra $M$ is weakly divisible and $\sigma$-complete if and only if $X$ is basically disconnected and $a\mapsto a^*$ is an MV-algebra isomorphism of $M$ onto  $C_1(X)$.
\end{theorem}

On the other hand,  for every constant $\alpha$, $0\leq \alpha  \leq 1$, $\alpha 1\in M$. It follows that the tribe $M^*$ contains all constant functions with values in $[0,1]$, an therefore by \cite[Theorem 8.4.1]{RiNe}, \cite[Theorem 7.1.7]{DvPu}, $M^*$ contains all ${\mathcal B}(M^*)$-measurable functions.

In the next lemma we show that if $N$ is a convex $\sigma$-sub-MV algebra of $M$\footnote{$N$ is a convex sub-MV-algebra of a convex MV-algebra $M$ iff  $N$ is a sub-MV-algebra of $M$ and $a\in N$, $\alpha \in [0,1]$ implies $\alpha a\in N$.} then $N^*$ consists of all ${\mathcal B}(N^*)$- measurable functions. The following setup is inspired by \cite{Sid}.

\begin{lemma}\label{le:sid} Let $M$ be a convex $\sigma$-MV-algebra and let $N$ be a convex $\sigma$-sub-MV-algebra
of $M$.  Let $N^*:=\{ f\in M^*: \eta(f)\in N\}$, and ${\mathcal B}(N^*)$ be the set of all characteristic functions belonging to $N^*$. Then $N^*$ consists exactly of all ${\mathcal B}(N^*)$-measurable functions in $M^*$.
\end{lemma}

\begin{proof}
From the properties of $N$ it easily follows that ${\mathcal B}(N^*)$ is an algebra of subsets of $X$. If $A_i\in {\mathcal B}(N^*)$ for $i=1,2,3,\ldots$, then $\bigcup_{i=1}^kA_i\in {\mathcal B}(N^*)$.
Since $\chi_{\bigcup_{i=1}^k A_i} \nearrow \chi_{\bigcup_{i=1}^{\infty}A_i}$  $\eta$-almost everywhere,
we have $\eta(\chi_{\bigcup_{i=1}^k A_i}) \nearrow \eta(\chi_{\bigcup_{i=1}^{\infty}A_i})$, and since $N$ is a $\sigma$-MV-algebra, $\eta(\chi_{\bigcup_{i=1}^{\infty}A_i})\in N$, hence $\bigcup_{i=1}^{\infty}A_i\in {\mathcal B}(N^*)$. This shows that ${\mathcal B}(N^*)$ is a $\sigma$-algebra of subsets of $X$.

We show that every function $f\in N^*$ is ${\mathcal B}(N^*)$-measurable. By definition, every characteristic function from $N^*$ is ${\mathcal B}(N^*)$-measurable.  Let us consider an arbitrary $f\in N^*$. Let $\alpha \in [0,1)$, and put $A:=\{ x\in X: f(x)> \alpha\}$. Define $g(x):=\max(f(x),\alpha)$,  then $g(x)=f(x)$ if $x\in A$, $g(x)=\alpha$ if $x\notin A$ and since $N^*$ is convex and closed under lattice operations, $g$ is also an element of $N^*$.
Since $g(x)\geq \alpha$, we have $h(x):=g(x)-\alpha\in N^*$, and $h(x)>0$ for $x\in A$, while $h(x)=0$ for $x\notin A$.

We denote $A_n:=\{ x\in X: h(x)> \frac{1}{n}\}$; then $A_1\subset A_2\subset \ldots$, $\bigcup_n A_n=A$.
Furthermore for each $n=1,2,\ldots$ we define the functions $g_n(x):=n.\min(h(x),\frac{1}{n})$. Clearly, $0\leq g_n(x)\leq 1$ and $g_n\in N^*$.

Evidently, $g_n(x)=1$ for $x\in A_n$, $g_n(x)=n.h(x)$ for $x\notin A_n$.  Since for $x\notin A$, $h(x)=0$, we see that $g_n(x) \nearrow  \chi_A(x)$ for all $x\in X$.  Since $N$ is a $\sigma$-MV algebra, this entails that $\chi_A\in N^*$.  This proves that all functions $f\in N^*$ are ${\mathcal B}(N^*)$-measurable.

On the other hand, going, as usual, from characteristic functions to simple functions etc., we show that $N^*$ contains all  ${\mathcal B}(N^*)$-measurable functions from $M^*$. Hence $N^*$ consists precisely of all ${\mathcal B}(N^*)$-measurable functions from $M^*$. In particular, if $f,g\in N^*$ then $f.g\in N^*$.
\end{proof}

\begin{corollary}\label{co:wd} Let $M$ be a convex $\sigma$-MV-algebra, $m$ a $\sigma$-additive state on $M$, and $N$ a convex $\sigma$-MV-subalgebra of $M$. Then for every $a\in M$, the conditional expectation $m(a|N)\in N^*$, and $\eta(m(a|N))$ belongs to $N$.
\end{corollary}

Let $M$ be a $\sigma$-MV-algebra, $N$ a $\sigma$-MV-subalgebra of $M$, and $m$ a $\sigma$-additive state on $M$. Then
 ${\mathcal I}_m:=\{ a\in M: m(a)=0\}$ is a $\sigma$-MV algebra ideal on $M$ and the quotient
$\tilde M:=M|\mathcal I_m$ is a $\sigma$-MV-algebra. Put $\tilde{N}:=\{ [a]:a\in N\}$, hence $\tilde{N}$ is the subclass of $\tilde{M}$ consisting of all equivalence classes having a representative in $N$.

\begin{theorem}\label{th:quot} Let $M$ be a convex $\sigma$-MV algebra, $N$ a convex $\sigma$-MV-subalgebra of $M$,  and $m$ a $\sigma$-additive state on $m$.  Define the mapping $\tau:\tilde{M} \to \tilde{M}$ by $\tau[a]:=[\eta(m(a|N))]$. Then $\tau$ is a strong state operator on $\tilde{M}$, and the range of $\tau$ is $\tau(\tilde{M})=\tilde{N}$.
\end{theorem}

\begin{proof} (1) First we prove that $\tau$ is well defined. That is, we have to prove that for any $a_1,a_2\in M$, $a_1\sim a_2$ implies $\eta(m(a_1|N))\sim \eta(m(a_2|N))$. Now $a_1\sim a_2$ iff $m(a_1\Delta a_2)=0$. Let $b\in {\mathcal B}(N)$ with $b^*=\chi_B$, $B\in {\mathcal B}(N^*)$. Then $(a_1\wedge b)\Delta(a_2\wedge b)= (a_1\Delta a_2)\wedge b$, so that $a_1\sim a_2$  entails that $m((a_1\wedge b)\Delta (a_2\wedge b))=0$, and hence $m(a_1\wedge b)=m(a_2\wedge b)$ for all $b\in {\mathcal B}(N)$. Write $f_1:=m(a_1|N), f_2:=m(a_2|N)$. Then $f_1,f_2$ are ${\mathcal B}(N^*)$-measurable, and we have
$\int_B(f_1-f_2)dP^*=0$, $\forall B\in {\mathcal B}(N^*)$. Putting $B_1=\{ \omega:f_1(\omega) > f_2(\omega)\}$, $B_2:=\{\omega: f_1(\omega) < f_2(\omega)\}$, we obtain $P^*(B_1)=0=P^*(B_2)$.
From this we have $P^*\{ \omega: |f_1-f_2|\neq 0\}=0 \, \implies \, m\circ \eta \{ \omega:f_1\Delta f_2 \neq 0\}=0\, \implies \, m\circ \eta(f_1\Delta f_2)=0\, \implies \, m(\eta(f_1)\Delta \eta(f_2))=0$, which yields $\eta(m(a_1|N))\sim \eta(m(a_2|N))$.

(2) Next we prove that $\tau$ is a $\sigma$-additive state operator. By Theorem \ref{th:4.2} (1), $\tau[1]=[1]$. By  Theorem \ref{th:4.2} (2) and (3), $a\boxdot b=0$ implies  $m(a\oplus b|N)=m(a|N)+m(b|N)\leq 1$ a.e. $[P^*]$, whence $\tau[a\oplus b]=\tau[a]\oplus \tau[b]$. Similarly, $\sigma$-additivity follows by Theorem \ref{th:4.2} (4). For every $[c]\in \tilde{N}$, there is $c_1\in N$ with $[c]=[c_1]$. Then by Lemma \ref{le:sid},  $c_1^*\in N^*$ is ${\mathcal B}(N^*)$-measurable, therefore $c_1^*=m(c_1|N)$, which implies $\tau[c]=[\eta(c_1^*)]=[c_1]$. Moreover, $\tau(\tau[c])=\tau[c_1]=[c_1]=\tau[c]$.  Conversely, for every $a\in M$, $m(a|N)\in N^*$, hence $\eta(m(a|N))\in N$, which gives $\tau[a]\in \tilde{N}$. It follows that $\tau$ is a state operator and  the range of $\tau$
is $\tilde{N}$. But $\tilde{N}$ is an MV-algebra, hence $\tau[a]\wedge \tau[b]\in \tilde{N}$, and therefore $\tau(\tau[a]\wedge \tau[b])=\tau[a]\wedge \tau[b]$, and hence $\tau$ is strong.
\end{proof}

Let $\tau$ be a $\sigma$-additive state operator on a convex $\sigma$-MV-algebra $M$. Then (by Theorem \ref{th:wd})
we may define a map $\tau^*:C_1(X)\to C_1(X)$ by
\begin{equation}\label{eq:t}
\tau^*(a^*):=\tau(a)^*
\end{equation}
We show that $\tau^*$ is a $\sigma$-additive state operator on $C_1(X)$. Indeed, $\tau^*(1^*)=\tau(1)^*=1^*$; $a^*+b^*\leq 1^*$ implies $(a^*+b^*)=(a\oplus b)^*$, whence $\tau^*(a^*+b^*)= \tau(a\oplus b)^*=\tau(a)^*+\tau(b)^*=\tau^*(a^*)+\tau^*(b^*)$; $\tau^*(\tau^*(a^*))=\tau^*(\tau(a)^*)=(\tau(\tau(a))^*=\tau(a)^*=\tau^*(a^*)$; moreover,  $a_n^*\nearrow a^*$ implies $a_n\nearrow a$, so that $\tau(a_n)\nearrow \tau(a)$, whence $\tau(a_n)^*\nearrow \tau(a)^*$, which yields $\tau^*(a_n^*)\nearrow \tau^*(a^*)$ (notice that by \cite[Lemma 9.12]{Goo}, the isomorphism $a\mapsto a^*$
 preserves countable suprema and infima).

Conversely, if $\tau^*$ is a state operator on $C_1(X)$, then $\tau(a):=\eta(\tau^*(a^*))$ defines a state operator on $M$.
In addition, from $\omega(a\wedge b)=\inf \{\omega(a),\omega(b)\}$ for every extremal state $\omega \in X$, we obtain $(a\wedge b)^*=a^*\wedge b^*$, which yields $\tau^*(\tau^*(a^*)\wedge \tau^*(b^*))=\tau^*(\tau(a)^*\wedge \tau(b)^*)=\tau^*((\tau(a)\wedge \tau(b))^*)=(\tau(\tau(a)\wedge \tau(b)))^*$, and it shows that $\tau$ is strong iff $\tau^*$ is strong.  Therefore, studying state operators on $M$ may be replaced by studying state operators on $C_1(X)$ and Theorem \ref{th:str_cond}
 applies, hence $\tau$ is a strong state operator on $M$ iff $\tau^*$ is a conditional expectation on $C_1(X)$, in the sense of Section \ref{sec:jca}.

In the next theorem we show that a strong state operator on $M$ yields an MV-conditional expectation on $M$.

\begin{theorem} Let $M$ be a convex  $\sigma$-MV-algebra  and $\tau$ a $\sigma$-additive strong state operator on $M$.  Then there is a convex $\sigma$-MV subalgebra $N$ of $M$ and a state $m$ on $M$ such that for every $a\in M$,
$\tau^*(a^*)$ is an MV-conditional expectation of $a$ with respect to $N$ in $m$.
\end{theorem}

\begin{proof}

 Assume that $M$ is a convex $\sigma$-MV-algebra. We may identify $M$ with $C_1(X)$, and define, for $a,b\in M$, $a.b:=\eta(a^*.b^*)$. As $\tau$ is strong, it is a conditional expectation, hence for all $a,b\in M$, $\tau(a.\tau(b))=\tau(a).\tau(b)$.

Let $s$ be a $\sigma$-additive state on $M$. Then $m:=s\circ \tau$ is also a $\sigma$-additive state on $M$, and $m\circ \eta$ is a $\sigma$-additive state on $M^*$. Let $\mu$ be the probability measure $\mu:=m^*|{\mathcal B}(M^*)$.
 Put $N:=\tau(M)$, then it is clear that $N$ is a convex $\sigma$-MV-subalgebra of $M$.  For all $a\in M$, $b\in {\mathcal B}(N)$,

$$
m(a\wedge b)=\circ \tau \circ \eta(a^*.b^*).
$$
Taking into account that by (\ref{eq:t}) $\tau(a))=\tau(\eta(a^*))=\eta(\tau^*(a^*))$, and that $b=\tau(b)$ implies $b^*=\tau^*(b^*)$, we obtain

\begin{eqnarray*}
s\circ \tau \circ \eta(a^*.b^*)&=&s\circ \eta(\tau^*(a^*.b^*))\\
&=&s\circ \eta(\tau^*(\tau^*(a^*.\tau^*(b^* ))))\\
&=&s\circ \tau\circ\eta((\tau^*(a^*).\tau^*(b^*)))\\
&=&\int_X\tau^*(a^*)b^*d\mu,
\end{eqnarray*}
which shows that $\tau^*(a^*)$ is an MV-conditional expectation of $a$.

\end{proof}

\end{document}